\documentclass[12pt]{article}
\usepackage{amsfonts}
\usepackage{epsfig}
\usepackage{graphics}
\usepackage{amsmath,amsthm}
\usepackage{amssymb}
\usepackage{inputenc}
\usepackage[T2A]{fontenc}
\def \ve{\varepsilon}
\def \be{\begin{equation}}
\def \ee{\end{equation}}

\theoremstyle{plain}
\newtheorem{thm}{Theorem}
\newtheorem{cor}[thm]{Corollary}
\newtheorem{lem}[thm]{Lemma}
\newtheorem{prop}[thm]{Proposition}
\newtheorem{rem}[thm] {Remark}

\font \twbbb= msbm10 scaled \magstep0                 

\font \tenbbb= msbm7 scaled \magstep0                 

\newfam\bbbfam             

\textfont\bbbfam=\twbbb \scriptfont\bbbfam=\tenbbb    %

\begin{document}

\footnotetext{Mathematics Subject Classification 2000: }

\numberwithin{equation}{section}
\theoremstyle{plain}
\newtheorem{Th}{Theorem}
\newtheorem{Lm}{Lemma}
\newtheorem*{Df*}{Definition}
\newtheorem{Cr}{Corollary}[section]
\renewcommand{\abstractname}{}

\author{\centerline{\bf A.~Piatnitski} \\  \centerline{Narvik University College, Postboks 385, 8505 Narvik, Norway,}\\
\centerline{P. N. Lebedev Physical Institute of RAS, 53, Leninski pr., Moscow 119991, Russia,}\\
\centerline{e-mail: andrey@sci.lebedev.ru}   \and \centerline{\bf V.~Rybalko} \\
\centerline{Mathematical Division, B.Verkin Institute for Low Temperature Physics} \\
\centerline{and Engineering of the NASU,
47 Lenin ave., Kharkov 61103, Ukraine,} \\
\centerline{e-mail: vrybalko@ilt.kharkov.ua}}

\title{On the first eigenpair of
singularly perturbed operators with oscillating coefficients}

\maketitle

\begin{abstract}
The paper deals with a Dirichlet spectral problem for a
singularly perturbed second order elliptic operator with rapidly oscillating locally periodic coefficients.
We study the limit behaviour of the first eigenpair (ground state) of this problem.
The main tool in deriving the limit (effective) problem is the viscosity solutions technique for Hamilton-Jacobi
equations.
The effective problem need not have a unique solution. We study the non-uniqueness issue in a particular case of zero
potential and construct the higher order term of the ground state asymptotics.

\end{abstract}

\section{Introduction}

Given a singularly perturbed elliptic operator of the form
\begin{equation}
\mathcal{L}_\ve u=\ve^2a^{ij}(x,x/\ve^\alpha)\frac{\partial^2 u} {\partial
x_i\partial x_j} +\ve b^{j}(x,x/\ve^\alpha)\frac{\partial
u}{\partial x_j} +c(x,x/\ve^\alpha) u \label{operator}
\end{equation}
with a small parameter $\ve>0$,
we consider a Dirichlet spectral problem
$$
\mathcal{L}_\ve u=\lambda u, \qquad u=0\ \ \ \hbox{on }\partial \Omega
$$
stated in a smooth bounded domain $\Omega\subset\mathbb R^N$.
We assume that the coefficients $a^{ij}(x,y)$, $b^{j}(x,y)$
and  $c(x,y)$ are 
sufficiently regular functions periodic in $y$ variable, and that
$a^{ij}(x,y)$ satisfy the uniform ellipticity condition.
Finally,  $\alpha>0$ is a fixed positive parameter.
Let us remark that in the underlying convection-diffusion model
$\ve$ represents characteristic ratio between the diffusion and convection coefficients,
while $\ve^\alpha$ refers to the microstructure period.

As well known, the operator $\mathcal{L}_\ve$ has a discrete spectrum, and
the first eigenvalue $\lambda_\ve$ (the eigenvalue with the
maximal real part) is real and simple; the corresponding
eigenfunction $u_\ve$ can be chosen to satisfy $u_\ve>0$ in
$\Omega$. The goal of this work is to study the asymptotic
behavior of $\lambda_\ve$ and $u_\ve$ as $\ve\to 0$.

The first eigenpair (ground state)
of (\ref{operator}) plays a crucial role when studying the large
time behavior of solutions to the corresponding parabolic initial
boundary problem. The first eigenvalue characterizes an
exponential growth or decay of a typical solution, as
$t\to\infty$, while the corresponding eigenfunction describes the
limit profile of a normalized solution.

Also, since in a typical case the first eigenfunction shows a singular behavior,
as $\ve\to0$, in many applications it is
important to know the set of concentration points of $u_\ve$, the
so-called hot spots. This concentration set might consist of one
point, or finite number of points, or a surface of positive
codimension, or it might have more complicated structure. An
interesting discussion on hot spots can be found in \cite{Rau75}.

Boundary value problems for singularly perturbed elliptic operators have
been widely studied in the existing literature. An important contribution to this topic has been
done in the classical work \cite{ViLu} that deals with singular perturbed operators
with smooth non-oscillating coefficients under the assumption that for $\ve=0$ the
problem remains (in a certain sense) well-posed.

The Dirichlet problem for a convection-diffusion operator with a small diffusion and with a convection directed outward
at the domain boundary was studied for the first time in \cite{WF}. The
approach developed in that work relies on large deviation results for trajectories of a diffusion process being a
solution of the corresponding stochastic differential
equation.

The probabilistic interpretation of solutions and the aforementioned large deviation principle have also been used in
\cite{Kif1_80}, \cite{Kif2_80}, \cite{Ki87}, where the first
eigenvalue is studied  for a second order elliptic operator being
a singular perturbation of a first order operator.

There are two natural approaches that can be used for studying the logarithmic asymptotics
of the principal eigenfunction  of a second order singularly perturbed operator.
One of them relies on the above mentioned large deviation results for diffusion processes with
a small diffusion coefficients.
This method was used in \cite{Pi} for studying operators with smooth coefficients on a compact
Riemannian manifolds.


We follow yet another (deterministic) approach
based on the viscosity solution techniques for nonlinear PDEs.
In the context of linear singularly perturbed equations, these techniques
were originally developed
in \cite{EI} and followed by  \cite{BP}, \cite{IK},
\cite{Pe}, \cite{CC}, \cite{CCS} and other works
(see also a review in \cite{Bar}).
Since $u_\ve>0$ in $\Omega$, we
can represent $u_\ve$ as $u_\ve(x)=e^{-W_\ve(x)/\ve}$ to find that
$W_\ve$ satisfies
\begin{equation}
\label{eqforW}
-\ve a^{ij}(x,x/\ve^\alpha)\frac{\partial^2 W_\ve}{\partial x_i\partial x_j}+
H(\nabla W_\ve, x,x/\ve^\alpha)=\lambda_\ve
\end{equation}
with $H(p,x,y)=a^{ij}(x,y)p_i p_j-b^j(x,y)p_j+c(x,y)$, and the
Dirichlet boundary condition for $u_\ve$ yields $W_\ve=+\infty$ on
$\partial\Omega$. 
Using perturbed test
functions we pass to the limit in (\ref{eqforW}) and get the limit
Hamilton-Jacobi equation of the form
\begin{equation}
\overline H(\nabla W(x),x)=\lambda \qquad \text{in}\ \Omega. \label{homprob}
\end{equation}
with an effective Hamiltonian $\overline H(p,x)$ whose definition depends on
whether $\alpha>1$, $\alpha=1$ or $0<\alpha<1$. We show that in
the limit $\ve\to 0$ the boundary condition $W_\ve=+\infty$ on
$\partial\Omega$ yields
\begin{equation}
\overline H(\nabla W(x),x) \geq \lambda\qquad \text{on}\ \partial\Omega.
\label{hombcond}
\end{equation}
 The latter  condition
is known \cite{So1}, \cite{CL} 
as the state constraint boundary condition. Both equation
(\ref{homprob}) and boundary condition (\ref{hombcond}) are
understood in viscosity sense.

Equations of type  (\ref{eqforW}) have been extensively studied in the existing literature.
One can find a short review of state of the art in \cite{LS}, \cite{Kos}  and in more recent works \cite{CS}, \cite{AS},
see also references therein. 

Earlier, singularly perturbed KPP-type reaction-diffusion equations were studied in \cite{MaSou} where, in particular, equations with rapidly oscillating coefficients were considered.
It was shown that the classical Huygens principle might fail
to work in this case.

In the present work,  deriving the effective problem
(\ref{homprob})-(\ref{hombcond}) relies on
the idea of perturbed test functions originally proposed in \cite{E1}.
We strongly believe that with the help of the techniques developed recently in
\cite{LS}, \cite{KRV}, \cite{LS1}, \cite{AS}  this result
can be extended to a more general almost periodic setting as well as
random stationary ergodic setting. In other words, the periodicity assumption can be replaced
with the assumption that the coefficients
in (\ref{operator}) are almost periodic or random statistically homogeneous and
ergodic with respect to the fast variable, at least in the case $\alpha=1$.
The case $\alpha\not=1$ looks more difficult and might require some extra assumptions.
We refer to \cite{S}, \cite{RT},  \cite{CSW}, \cite{LS}, \cite{KRV}, \cite{CS}, \cite{AS} for
(far not complete list of)
various results on almost periodic and random
homogenization of nonlinear PDEs. However, the essential novelty of
this work comes in the (logically) second part of the paper devoted
to the improved ground state asymptotics and resolving the non-uniqueness issue
for (\ref{homprob})-(\ref{hombcond}). The generalization of
this part to non-periodic settings is an open problem.

Problem (\ref{homprob})--(\ref{hombcond}) is known as ergodic or
additive eigenvalue problem. Its solvability was first proved in \cite{LPV} in periodic setting,
more recent results are contained, e.g., in \cite{IM} as well as in
\cite{DS}, \cite{AS1} where stationary ergodic Hamiltonians were considered.
There exists the unique additive
eigenvalue $\lambda$ of (\ref{homprob})--(\ref{hombcond}) while the
eigenfunction $W$ need not be
unique even up to an additive constant. This non-uniqueness
issue is intimately related to the structure of the so-called
Aubry set of effective Hamiltonian which play the role of a hidden
boundary for (\ref{homprob})--(\ref{hombcond}). Loosely speaking
the non-uniqueness in (\ref{homprob})--(\ref{hombcond}) appears
when the Aubry set is disconnected. By contrast, for every $\ve>0$
the eigenfunction $u_\ve$ is unique up to a normalization, and it
is natural to find true limit of $W_\ve=-\ve \log u_\ve$ among
solutions of (\ref{homprob})--(\ref{hombcond}).
This challenging problem is addressed in a particular case of
(\ref{operator}) with $c(x,y) = 0$, $\alpha\geq1$. Following
\cite{Pi83} we introduce the effective drift (convection) and
 assume that it has a
finite number of hyperbolic fixed points in $\Omega$, and that the
Aubry set of the effective Hamiltonian coincides with this finite
collection of points. It follows from our results that in this
case $\lambda_\ve$ tends to zero as $\ve\to 0$. We show that
$\lambda_\ve/\ve$ has a finite limit that can be determined in
terms of eigenvalues of Ornstein-Uhlenbeck operators in $\mathbb
{R}^N$ obtained via local analysis of (\ref{operator}) at the
scale $\sqrt{\ve}$ in the vicinity of aforementioned fixed points.
This, in turn, enables fine selection of the additive
eigenfunction corresponding to $\lim_{\ve\to 0} W_\ve$.

\section{Main results}

We begin with standing hypotheses which are assumed to hold
throughout this paper. We assume that $\Omega$ is connected and
has $C^2$ boundary $\partial\Omega$; the coefficients
$a^{ij}(x,y),\, b^{j}(x,y),\, c(x,y) \in C^1(\overline{\Omega}\times\mathbb{R}^N)$ are
$Y$-periodic in $y$ functions, where $Y=(0,1)^N$. The matrix
$(a^{ij})_{i,j=\overline{1,N}}$ is uniformly positive
definite:
\begin{equation}
\label{ellipt}
a^{ij}(x,y)\zeta_i\zeta_j\geq m |\zeta|^2>0 \qquad
\forall \zeta\not=0,
\end{equation}
and, without loss of generality, we can assume the symmetry
$a^{ij}=a^{ji}$.

The first eigenfunction $u_\ve$ of the operator (\ref{operator})
can be normalized to satisfy
\begin{equation}
1=\max_{\Omega} u_\ve\ \  (u_\ve>0\ \text{in}\ \Omega),
\label{normalization}
\end{equation}
 then its scaled logarithmic transformation
\begin{equation*}
W_\ve:=-\ve\log u_\ve \label{logtransf}
\end{equation*}
is a nonnegative function vanishing at the points of maxima of
$u_\ve$.


The asymptotic behavior of $\lambda_\ve$ and $W_\ve$ is described
in

\begin{thm}
\label{mainth} The eigenvalues $\lambda_\ve$ converge as $\ve\to
0$ to the limit $\lambda$, which is the unique real number for which
problem
(\ref{homprob}),
(\ref{hombcond})
has a continuous viscosity solution. The functions $W_\ve$ converge
(up to extracting a subsequence) to a limit  $W$ uniformly on
compacts in $\Omega$, and every limit function $W$ is a viscosity
solution of (\ref{homprob}), (\ref{hombcond}).


The effective Hamiltonian $\overline H(p,x)$ in (\ref{homprob}) is
given by the following formulas, depending on the parameter
$\alpha$.

\medskip\noindent
(i)
If $\alpha>1$ then 
\begin{equation}
\label{effH1} \overline H(p,x)=\int_Y H(p,x,y)\vartheta(y)\,{\rm d} y
\end{equation}
where
$$
H(p,x,y)=a^{ij}(x,y)p_ip_j-b^j(x,y)p_j+c(x,y),
$$
and $\vartheta(y)$ is the unique $Y-$periodic solution of the
equation $\frac{\partial^2}{\partial y_i\partial
y_j}(a^{ij}(x,y)\vartheta)=0$ normalized by
$
\int_{Y}\vartheta(y)\, {\rm d}y=1.
$

\medskip\noindent
(ii)
If $\alpha=1$ then $\overline H(p,x)$ is the first
eigenvalue (eigenvalue with the maximal real part) of the problem
\begin{equation}
\begin{array}{l}
\displaystyle a^{ij}(x,y)\frac{\partial^2 \vartheta}{\partial y_i\partial y_j}
+(b^{j}(x,y)-2a^{ij}(x,y) p_i)\frac{\partial\vartheta}{\partial
y_j}+ H(p,x,y)\vartheta=\overline
H(p,x)\vartheta, \\
\vartheta(y)\quad \text{is}\
Y\text{-periodic}.
\end{array}
\label{effproblresonans}
\end{equation}
 According to the Krein-Rutman theorem
$\overline H(p,x)$ is real.

\medskip\noindent
(iii)
If $0<\alpha<1$ then
$\overline H(p,x)$ is the unique number such that the problem
\begin{equation}
\label{effH2}
H(p+\nabla \vartheta(y),x,y) =\overline H(p,x)
\end{equation}
has a $Y-$periodic viscosity solution $\vartheta(y)$;
here $p\in \mathbb R^N$ and $x\in\overline\Omega$ are parameters.
\end{thm}

We note that the effective Hamiltonian $\overline H(p,x)$ is
continuous on $\mathbb{R}^N\times\overline{\Omega}$, convex in $p$
and coercive, moreover $\overline H(p,x)\geq m_1 |p|^2- C$,
$m_1>0$. The viscosity solutions theory for such Hamiltonians is
well established. Following \cite{IM} (see also \cite{M}) we
present various representation formulas for the solutions of
problem (\ref{homprob})-(\ref{hombcond}).

Let us rewrite problem (\ref{homprob})-(\ref{hombcond}) in the
form
\begin{equation}
\overline H(\nabla W(x),x)\leq \lambda \qquad \text{in}\ \Omega
\label{homprob1}
\end{equation}
\begin{equation}
\overline H(\nabla W(x),x)\geq \lambda\qquad \text{in}\ \overline{\Omega},
\label{hombcond1}
\end{equation}
i.e. (\ref{homprob1}) requires that $W$ is a viscosity subsolution
in $\Omega$ while (\ref{hombcond1}) means that $W$ is a viscosity
supersolution in $\overline{\Omega}$. Then the number
$\lambda_{\overline H}$ (additive eigenvalue) for which
(\ref{homprob})-(\ref{hombcond}) has a solution is given by
\begin{equation}
\label{additiveeigenvalue} \lambda_{\overline H}=\inf\{\lambda;\,
\text{(\ref{homprob1}) has a solution}\ W\in
C(\overline{\Omega})\}.
\end{equation}
It can also be expressed in terms of action minimization,
$$
\lambda_{\overline H}=-\lim_{t\to \infty}\frac{1}{t}\inf \int_0^t
\overline L(\dot\eta,\eta)\,{\rm d}\tau,
$$
where the infimum is taken over absolutely continuous  curves
$\eta:[0,t]\to \overline{\Omega}$, and $\overline L(v,x)$ is the
Legendre transform of $\overline H(p,x)$,
$$
\overline L(v,x)=\max\{v\cdot p-\overline H(p,x)\}.
$$

Let us define now the distance function
\begin{equation}
\label{distancefunction} d_{\overline H-\lambda_{\overline
H}}(x,y)=\sup\{W(x)-W(y);\, W\in C(\overline{\Omega})\ \text{is a
solution of (\ref{homprob1}) for}\ \lambda=\lambda_{\overline H}\}.
\end{equation}
It is known (see, e.g., \cite{IM})
that $d_{\overline
H-\lambda_{\overline H}}(x,x)=0$, $d_{\overline
H-\lambda_{\overline H}}(x,y)$ is Lipschitz continuous,
$d_{\overline H-\lambda_{\overline H}}(x,y)\leq d_{\overline
H-\lambda_{\overline H}}(x,z)+ d_{\overline H-\lambda_{\overline
H}}(z,y)$. Besides, for every $y\in\overline{\Omega}$ the function
$d_{\overline H-\lambda_0}(x,y)$ is a solution of (\ref{homprob1})
for $\lambda=\lambda_{\overline H}$ and $\overline H(\nabla_x
d_{\overline H-\lambda_{\overline H}}(x,y),
x)\geq\lambda_{\overline H}$ in $\overline{\Omega}
\setminus\{y\}$. The number  $\lambda_{\overline H}$ is such that
the {\it Aubry} set $\mathcal{A}_{\overline H-\lambda_{\overline H}}$,
\begin{equation}
\label{Aubry} \mathcal{A}_{\overline H-\lambda_{\overline
H}}=\{y\in\overline{\Omega};\, d_{\overline H-\lambda_{\overline H}}(x,y)\
\text{is a solution of (\ref{hombcond1}) for} \
\lambda=\lambda_{\overline H}\},
\end{equation}
is nonempty. Note also that the distance function $d_{\overline
H-\lambda_{\overline H}}(x,y)$ admits the representation
\begin{equation}
\label{distancefunction1} d_{\overline H-\lambda_{\overline
H}}(x,y)=\inf\Bigl\{\int_0^t (\overline
L(\dot\eta,\eta)+\lambda_{\overline H}){\rm d}\tau,\,
\eta(0)=y,\eta(t)=x, t>0\Bigr\},
\end{equation}
and the Aubry set can be characterized by
\begin{equation}
\label{AubryVariational}
 y\in \mathcal{A}_{\overline H-\lambda_{\overline
H}}\Longleftrightarrow \sup_{\delta>0} \inf\Bigl\{\int_0^t (\overline
L(\dot\eta,\eta)+\lambda_{\overline H})\,{\rm d}\tau,\,
\eta(0)=\eta(t)=y,t>\delta\Bigr\}=0.
\end{equation}
The infimum in (\ref{distancefunction1}) and
(\ref{AubryVariational}) is taken over absolutely continuous
curves $\eta:[0,t]\to \overline{\Omega}$.

According to the definition of $d_{\overline H-\lambda_{\overline
H}}(x,y)$ every solution $W$ of (\ref{homprob})-(\ref{hombcond})
satisfies $W(x)-W(y)\leq d_{\overline H-\lambda_{\overline
H}}(x,y)$; this inequality holds, in particular,
 for all $x,y\in\mathcal{A}_{\overline
H-\lambda_{\overline H}}$. Conversely, given a function $g(x)$ on
$\mathcal{A}_{\overline H-\lambda_{\overline H}}$ which satisfies the
compatibility condition $g(x)-g(y)\leq d_{\overline H-\lambda_{\overline
H}}(x,y)$ $\forall x,y\in \mathcal{A}_{\overline H-\lambda_{\overline H}}$
then
\begin{equation}
\label{represnt_forH} W(x)=\min\{d_{\overline H-\lambda_{\overline
H}}(x,y)+g(y);\, y\in\mathcal{A}_{\overline H-\lambda_{\overline H}}\}
\end{equation}
is the unique solution of (\ref{homprob})-(\ref{hombcond}) for
$\lambda=\lambda_{\overline H}$ such that $W(x)=g(x)$ on
$\mathcal{A}_{\overline H-\lambda_{\overline H}}$. In Appendix~\ref{A}
we
show the following simple uniqueness criterion for problem
(\ref{homprob})-(\ref{hombcond}): a solution $W$ (for
$\lambda=\lambda_{\overline H}$) is unique up to an additive
constant if and only if $S_{\overline H-\lambda_{\overline
H}}(x,y)=0$ $\forall x,y\in\mathcal{A}_{\overline
H-\lambda_{\overline H}}$, where $S_{\overline
H-\lambda_{\overline H}}(x,y)$ denotes the symmetrized distance,
$S_{\overline H-\lambda_{\overline H}}(x,y)=d_{\overline
H-\lambda_{\overline H}}(x,y)+d_{\overline H-\lambda_{\overline
H}}(y,x)$.

The interesting issue of non-uniqueness in the limit (homogenized)
problem can be resolved by studying next terms in the
asymptotic expansion of $\lambda_\ve$. This question is rather
complicated, and we mainly focus in this work on a particular case when
$c(x,y)=0$ and $\alpha =1$, so that operator (\ref{operator}) takes the form
\begin{equation}
\label{operator-redu}
\mathcal{L}_\ve u=\ve^2a^{ij}(x,x/\ve)\frac{\partial^2 u} {\partial
x_i\partial x_j} +\ve b^{j}(x,x/\ve)\frac{\partial
u}{\partial x_j}.
\end{equation}
Moreover we assume that
$\lambda_{\overline H}=0$ and that the corresponding Aubry set
$\mathcal{A}_{\overline H}$ has a special structure. The
analogous result for $\alpha>1$ is established in Section \ref{Sec8}.

For $\alpha\geq 1$ the effective Hamiltonian $\overline
H(p,x)$ is strictly convex in $p$,
 i.e.
$\left(\frac{\partial^2}{\partial p_i \partial p_j}\overline
H(p,x)\right)_{i,j=\overline{1,N}}$ is positive definite for all
$p\in\mathbb R^N$ and $x\in\overline\Omega$, see \cite{C},
or \cite{DP} for $\alpha=1$, while for $\alpha>1$ the
Hamiltonian $\overline H(p,x)$ is a quadratic function in $p$. Note also
that  if $c(x,y)=0$ then $\overline H(0,x)=0$.
Therefore,
the Lagrangian $\overline L(v,x)$ is strictly convex
and $\overline
L(v,x)=\max \{p\cdot v-\overline H(p,x)\}\geq -\overline
H(0,x)=0$, in the case $c(x,y)=0$, $\alpha=1$ we are interested in.  Thus we have
$$
\overline L(v,x)\geq 0, \quad \text{and}\  \overline L(v,x)=0\,
\Longleftrightarrow\, v_j=\frac {\partial \overline H}{\partial
p_j}(0,x).
$$
On the other hand direct calculations show that
\begin{equation}
\label{effectivedrift}
 -\frac {\partial \overline H}{\partial
p_j}(0,x)=\overline b^j(x):=\int_Y b^j(x,y)\theta^\ast(x,y){\rm d}y,
\end{equation}
the functions $\overline b^j(x)$ being components of the so-called
effective drift $\bar b(x)$  defined by the right hand side of
(\ref{effectivedrift}) via the $Y$-periodic solution $\theta^\ast$
of
\begin{equation}
\frac{\partial^2}{\partial y_i\partial y_j}
\Bigl(a^{ij}(x,y)\theta^\ast\Bigr) -\frac{\partial}{\partial
y_j}\Bigl(b^{j}(x,y)\theta^\ast\Bigr)=0
\label{effdualproblresonans}
\end{equation}
normalized by $\int_Y \theta^\ast{\rm d}y=1$. (Note that
$\theta^\ast>0$ and it is a $C^2$ function.) Thus the Lagrangian
$\overline L(v,x)$ can be represented in the form
$$
\overline L(v,x)=\kappa \sum(v_j+\overline b^j(x))^2+\tilde
L(v,x),\quad \text{where} \ 0\leq \tilde L(v,x)\leq \tilde\kappa\sum
(v_j+\overline b^j(x))^2,\ 0<\kappa<\tilde \kappa.
$$
This implies, in view of (\ref{AubryVariational}), that the Aubry
set $\mathcal{A}_{\overline H}$ of the Hamiltonian $\overline H$
coincides with that of the Hamiltonian $\sum p_j^2-\overline
b^j(x) p_j$ whose corresponding Lagrangian is $\frac{1}{4}\sum
(v_j+\overline b^j(x))^2$. In particular, the additive eigenvalue
$\lambda_{\overline H}$ is zero if and only if there is an orbit
$\eta:\mathbb{R}\to\overline{\Omega}$, $\dot\eta=-\overline
b(\eta)$. We moreover assume that
%
%
\begin{equation}
\label{fixedpoints}
\begin{array}{l}
\mathcal{A}_{\overline H}\not=\emptyset \ \text{and}\
\mathcal{A}_{\overline H}\subset \Omega,
\\
\mathcal{A}_{\overline H}\ \text{is a finite set of hyperbolic
fixed points $\xi$ of the ODE $\dot x=-\overline b(x)$}.
\end{array}
\end{equation}
Under this assumption we are able to study the leading (of
order $\ve$) term of the asymptotic expansion of $\lambda_\ve$.
This in turn allows us to establish a sufficient condition for
selecting the unique limit of functions $W_\ve$ among solutions of
the homogenized problem (\ref{homprob}), (\ref{hombcond}).

\begin{thm}
\label{thwithoutdissipation} Let $\alpha=1$ and $c(x,y)=0$. Then,
under condition (\ref{fixedpoints}) we have
\begin{equation}
\lambda_\ve=\ve \overline \sigma+\bar o(\ve),\qquad \text{where}\
\overline\sigma=\max \{\sigma(\xi);\,
\xi\in\mathcal{A}_{\overline H}\}, \label{selection}
\end{equation}
and $\sigma(\xi)$ is the sum of negative real parts of the
eigenvalues of the matrix
$$
-B^{ij}(\xi)=-\frac{\partial \overline b^j}{\partial x_i}(\xi)
$$
corresponding to the linearized effective drift at $\xi$ (since
every fixed point $\xi$ is assumed to be hyperbolic $-B(\xi)$ has
no eigenvalues with zero real part). Moreover, if the maximum in
(\ref{selection}) is attained at exactly one $\xi=\overline{\xi}$ then

\medskip\noindent
(i) the scaled logarithmic transformations $W_\ve=-\ve\log u_\ve$ of
eigenfunctions $u_\ve$ (normalized by (\ref{normalization}))
converge uniformly on compacts in $\Omega$ to
 $W(x)=d_{\overline H}(x,\overline{\xi})$,
i.e. $W$ is the maximal viscosity solution of $\overline H(\nabla
W(x),x)=0$ in $\Omega$, $\overline H(\nabla W(x),x)\geq 0$ on
$\partial\Omega$, such that $W(\overline\xi)=0$;

\medskip\noindent
(ii)
$u_\ve(\overline{\xi}+\sqrt{\ve} z)\to u(z)$
in $C(K)$ and weakly in $H^1(K)$ for every compact $K$, and the limit $u$ is the
unique positive eigenfunction of the Ornstain-Uhlenbeck operator,
\begin{equation}
\label{orst_uhl}
Q^{ij}\frac{\partial^2 u}{\partial z_i\partial z_j}+z_iB^{ij}\frac{\partial u}{\partial z_j}
 =\overline{\sigma} u \qquad \text{in}\ \mathbb{R}^N,
\end{equation}
normalized by $u(0)=1$ and satisfying the following condition, $u(z)e^{\mu|\Pi_s^\ast
z|^2-\nu |\Pi_u^\ast z|^2}$ is bounded on $\mathbb{R}^N$
for some $\mu>0$ and every $\nu>0$.
The coefficients in (\ref{orst_uhl}) are given by  $B^{ij}=B^{ij}(\overline \xi)$,
$Q^{ij}=\frac{1}{2}\frac{\partial^2 \overline{H}}{\partial p_i\partial p_j}(0,\overline{\xi})$;
$\Pi_s$ and $\Pi_u$ denote spectral projectors on the invariant
subspaces of the matrix $B$ corresponding to the eigenvalues with
positive and negative real parts (stable and unstable subspaces of
the system $\dot z_i=-B^{ij}z_j$).
\end{thm}

\begin{rem}\label{perturbedAubry}
Condition (\ref{fixedpoints}) is satisfied, in
particular, when the vector field $b(x,y)$ is a $C^1$-small
perturbation of a gradient field $\nabla P(x)$ with $C^2$
potential $P(x)$ having the following properties: the set $\{x\in
\overline \Omega;\, \nabla P(x)=0\}$ is formed by a finite
collection of points in $\Omega$ and the Hessian matrix
$\left(\frac{\partial^2}{\partial x_i \partial
x_j}P(x)\right)_{i,j=\overline{1,N}}$ at every such a point is
nonsingular (see Appendix \ref{B}).

\medskip\noindent
Condition (\ref{fixedpoints}) is satisfied if and only if the vector field $\overline b$ possesses the following properties:

\smallskip\noindent
$\bullet$
$\overline b$ has a finite number of
fixed points
in $\overline\Omega$, say $\xi^1,\dots \xi^n$. All of them are hyperbolic, and none of them is situated on $\partial\Omega$.

\smallskip\noindent
$\bullet$
$\forall \,y\in\overline\Omega$,  either
$\sup\{t<0\,:\,x^y(t)\not\in\overline\Omega\}>-\infty$, or
$\lim\limits_{t\to-\infty}x^y(t)=\xi^j$ for some $j\in\{1,\dots,n\}$, where
$x^y$ is a solution of the ODE \ $\dot x^y=-\overline b(x^y)$, $x^y(0)=y$.

\smallskip\noindent
$\bullet$
there is no any closed path
$\xi^{j_1},\xi^{j_2},\dots,\xi^{j_k}=\xi^{j_1}$ with $k\geq 2$ such that for any two consecutive points $\xi^{j_s}$ and
$\xi^{j_{s+1}}$ there is a solution of the equation $\dot x=-\overline b(x)$ with
$\lim\limits_{t\to -\infty}x(t)=\xi^{j_s}$
and $\lim\limits_{t\to +\infty}x(t)=\xi^{j_{s+1}}$. Note that $\xi^{j_1}$ might coincide with $\xi^{j_2}$.
\end{rem}

\begin{rem} It is not hard to show that under
condition (\ref{fixedpoints}) we have $S_{\overline
H}(\xi,\xi^\prime)>0$ $\forall \xi,\xi^\prime\in
\mathcal{A}_{\overline H}$, $\xi\not=\xi^\prime$. This means that
problem (\ref{homprob}), (\ref{hombcond}) does have many solutions
unless $\mathcal{A}_{\overline H}$ is a single point.
\end{rem}

Note that condition (\ref{fixedpoints}) of Theorem \ref{thwithoutdissipation} assumes, in particular,
that all $\omega$(and $\alpha$)-limit
points of the ODE $\dot x=-\overline{b}(x)$ are fixed points. Another important case,
 when the ODE $\dot x=-\overline{b}(x)$ has
limit cycles in $\Omega$ (which is also the case of general position),
will be considered in a separate paper.



\section{Singularly perturbed operators in the periodic setting}

Consider the spectral problem for singularly perturbed elliptic operators of the form
\label{periodS}
\begin{equation}
\label{operper}
\mathcal{L}_\ve^{(per)}u=\ve^2a^{ij}(x)\frac{\partial^2 u}
{\partial x_i\partial x_j} +\ve
b^{j}(x)\frac{\partial u}{\partial x_j} +c(x) u,
\end{equation}
with $Y$-periodic coefficients $a^{ij},\, b^j,\, c\in C^1(\mathbb{R}^N)$,  $u$ also
being $Y$-periodic. We assume the uniform ellipticity condition
$a^{ij}(x)\zeta_i\zeta_j\geq m |\zeta|^2>0$ $\forall
\zeta\in\mathbb{R}^N\setminus\{0\}$ and the symmetry
$a^{ij}=a^{ji}$. Similarly to the case of the Dirichlet boundary
condition, the first eigenvalue $\mu_\ve$ of $\mathcal{L}_\ve^{(per)}$
(eigenvalue with the maximal real part) is a real and simple
eigenvalue, the corresponding eigenfunction $u_\ve$ can be chosen
to satisfy $0< u_\ve(x)\leq \max u_\ve=1$. The asymptotic behavior
of $\mu_\ve$ as $\ve\to 0$ and $u_\ve$ was studied in \cite{Pi}
using a combination of large deviation and variational techniques.
We recover hereafter the results of \cite{Pi} by means of
vanishing viscosity approach and establish as a bi-product some
bounds for derivatives of functions
$
W_\ve(x)=-\ve \log u_\ve(x)
$
that are essential in the proof of Theorem \ref{mainth}.

First we derive the a priori bounds for the eigenvalues.

\begin{lem}
\label{estforLambdaper} For every $\ve>0$ the eigenvalue $\mu_\ve$
of $L^{(per)}_\ve$ satisfies the inequalities
\begin{equation}
\label{boundsforperidiclambda} \min c(x)\leq \mu_\ve\leq \max
c(x).
\end{equation}
\end{lem}

\begin{proof} Let $x^\prime$ be a maximum point of $u_\ve$, we have
$$
\nabla u_\ve(x^\prime)=0,\qquad
\ve^2a^{ij}(x^\prime)\frac{\partial^2 u_\ve}
{\partial x_i\partial x_j}(x^\prime) \leq 0,
$$
therefore $c(x^\prime) u_\ve(x^\prime)\geq\mu_\ve
u_\ve(x^\prime)$, i.e. $\mu_\ve\leq \max c(x)$. Similarly, if
$x^{\prime\prime}$ is a minimum point of $u_\ve$ then $\mu_\ve
u_\ve(x^{\prime\prime})\geq
c(x^{\prime\prime})u_\ve(x^{\prime\prime})$ and therefore
$\mu_\ve\geq \min c(x)$.
\end{proof}

Since $u_\ve=e^{-W_\ve(x)/\ve}$ we have
\begin{equation}
\label{eqforWper} -\ve a^{ij}(x)\frac{\partial^2 W_\ve}{\partial
x_i\partial x_j}+ a^{ij}(x)\frac{\partial W_\ve}{\partial
x_i}\frac{\partial W_\ve}{\partial x_j}- b^j(x)\frac{\partial
W_\ve}{\partial x_j}+c(x)=\mu_\ve.
\end{equation}
The bounds for the first and second derivatives of
 $W_\ve(x)$ are obtained in the following

\begin{lem}
\label{estimatesWper}  There is a constant $C$, independent of $\ve$, such
that
\begin{equation}
\label{derivatives}
\max |\nabla W_\ve|\leq C, \qquad  \max |\partial^2 W_\ve/\partial x_i\partial x_j|\leq C/\ve.
\end{equation}
\end{lem}

\begin{proof} The proof of the first bound in (\ref{derivatives}) is borrowed from \cite{EI}.
Let $D_1(x):=|\nabla W_\ve(x)|^2$ and
$D_2(x):=\sum |\partial^2 W_\ve(x)/\partial x_i\partial x_j|^2$ . From (\ref{eqforWper}) in
conjunction with (\ref{boundsforperidiclambda}) we get $m D_1\leq C(\ve D_2^{1/2}+D_1^{1/2}+1)$,
this in turn implies that
\begin{equation}
\label{Dapriori}
D_1\leq C(\ve D_2^{1/2}+1).
\end{equation}

Assume that $D_1$ attains its maximum at a point $x^\prime$, then
we have $\nabla D_1(x^\prime)=0$ and $a^{ij}(x^\prime)\frac{\partial^2
D_1}{\partial x_i\partial x_j}(x^\prime)\leq 0$ or
\begin{equation}
\label{zerograd}
\frac{\partial^2 W_\ve}{\partial x_i\partial x_k}(x^\prime)\frac{\partial W_\ve}{\partial x_k}(x^\prime)=0
\end{equation}
and
\begin{equation}
\label{maxequal}
\ve\sum_k a^{ij}\frac{\partial^2 W_\ve}{\partial x_i\partial x_k}
\frac{\partial^2 W_\ve}{\partial x_j\partial x_k}\leq -
\ve\sum_k a^{ij}\frac{\partial^3 W_\ve}{\partial x_i\partial x_j\partial x_k}\frac{\partial W_\ve}{\partial x_k} \qquad
\text{at}\ x^\prime.
\end{equation}
In order to bound the right hand side of  (\ref{maxequal}) we take
derivatives of (\ref{eqforWper}), this yields
\begin{equation}
-\ve a^{ij}\frac{\partial^3 W_\ve}{\partial x_i\partial x_j\partial x_k}
=\ve \frac{\partial a^{ij}}{\partial x_k}\frac{\partial^2 W_\ve}{\partial x_i\partial x_j}
-2a^{ij}\frac{\partial^2 W_\ve}{\partial x_i\partial x_k}\frac{\partial W_\ve}{\partial x_j} +b^i
\frac{\partial^2 W_\ve}{\partial x_i\partial x_k}+\frac{\partial b^{i}}{\partial x_k}
\frac{\partial W_\ve}{\partial x_i}
-\frac{\partial c}{\partial x_k}.
\label{equationforderivatives}
\end{equation}
Then we multiply (\ref{equationforderivatives}) by $\partial W_\ve/\partial x_k$,
sum up the equations in $k$ and insert the result into (\ref{maxequal}) to obtain
\begin{equation*}
\ve m D_2(x^\prime)\leq
\ve\sum_k a^{ij}(x^\prime)\frac{\partial^2 W_\ve}{\partial x_i\partial x_k}(x^\prime)
\frac{\partial^2 W_\ve}{\partial x_j\partial x_k}(x^\prime)\\
\leq   C
\Big(\ve D_1^{1/2}(x^\prime)D_2^{1/2}(x^\prime)  + D_1(x^\prime)+D_1^{1/2}(x^\prime)\Big).
\end{equation*}
Next we use (\ref{Dapriori}) to get that $D_2(x^\prime)\leq C/\ve$, and
exploiting once more (\ref{Dapriori}) we obtain the first bound in (\ref{derivatives}).

To show the second bound in (\ref{derivatives}) we use the following interpolation inequality
\begin{equation}
\label{interpolation}
\|\nabla u\|_{L^\infty}^2\leq
C(\|a^{ij}\partial^2 u/\partial x_i\partial x_j\|_{L^\infty}+\|u\|_{L^\infty})\|u\|_{L^\infty},
\end{equation}
which holds for every $Y$-periodic $u$ with a constant $C$
independent of $u$.
The proof of this inequality follows the lines
of one in the Appendix of \cite{BBH} (here it is important that the coefficients $a^{ij}$ are
Lipschitz continuous).
We apply (\ref{interpolation}) to (\ref{equationforderivatives}) to obtain
\begin{equation}
\label{resultinequality}
\|\partial^2 W_\ve/\partial x_l\partial x_k\|_{L^\infty}^2 \leq \frac{C}{\ve}
(\sum \|\partial^2 W_\ve/\partial x_i\partial x_j\|_{L^\infty}+1)\qquad \forall l,k,
\end{equation}
here we have also used the first bound in (\ref{derivatives}).
From (\ref{resultinequality}) one easily derives the second bound in (\ref{derivatives}).
\end{proof}

It follows from Lemma \ref{estforLambdaper} that $\mu_\ve\to\mu$, up to extracting
a subsequence. Due to Lemma~\ref{estimatesWper}
the family of functions $W_\ve(x)$ is equicontinuous, moreover
$\min W_\ve(x) =0$ therefore passing to a further subsequence (if
necessary) we have  $W_\ve(x)\to W(x)$ uniformly. The
standard arguments show that the pair $\mu$ and $W$ satisfies the
equation
\begin{equation}
\label{viscperiod} a^{ij}(x)\frac{\partial W}{\partial
x_i}\frac{\partial W}{\partial x_j} -b^{j}(x)\frac{\partial
W}{\partial x_j}+c(x)=\mu
\end{equation}
in the viscosity sense.The number $\mu$ for which (\ref{viscperiod}) has a periodic
viscosity solution is unique (see \cite{LPV},\cite{E}), therefore
the entire sequence $\mu_\ve$ converges to $\mu$ as $\ve\to 0$.

\section{A priori bounds}

In this section we show that the eigenvalues $\lambda_\ve$ of
(\ref{operator}) are uniformly bounded  and the functions $W_\ve$
(given by (\ref{logtransf})) uniformly converge on compacts in
$\Omega$ as $\ve\to 0$, up to extracting a subsequence. We also
prove an auxiliary result on the behavior of the minimum points of
$W_\ve-\phi$ (where $\phi$ is an arbitrary $C^2$ function) which
is important in the subsequent analysis.

Because of the Dirichlet boundary condition on the boundary
$\partial\Omega$ and fast oscillations of the coefficients the
arguments here are more involved than those in the periodic case.


\begin{lem}
\label{boundsforlambda}
There is a constant $\Lambda$ independent of $\ve$ and such that
\begin{equation}
\label{boundsfolambdaDir} -\Lambda \leq \lambda_\ve\leq\sup c(x,y).
\end{equation}
\end{lem}

\begin{proof}
The proof of the upper bound follows by the maximum principle as
in Lemma \ref{estforLambdaper}.

To derive a lower bound for $\lambda_\ve$ we construct a function
$v_\ve$ and choose a number $\Lambda>0$ such that
$v_\ve=0$ on $\partial\Omega$, and
\begin{equation}
\label{testLowerB_eigenval}
\mathcal{L}_\ve v_\ve-\lambda v_\ve>0
\qquad  \text{in}\ \Omega
\end{equation}
for every $\lambda<-\Lambda$, $0<\ve<1$. There is a
function $W\in C^2(\overline{\Omega})$ satisfying the following
conditions, $W>0$ in $\Omega$ and  $W=0$ on $\partial \Omega$,
$|\nabla W|> 1$ in a neighborhood of $\partial \Omega$. Set
$v_\ve(x):=e^{\kappa W(x)/\ve}-1$, where $\kappa $ is a positive parameter to
be chosen  later. We assume that $-\Lambda\leq \min c(x,y)$ so that $\lambda<\min c(x,y)$. Then we
have
\begin{equation*}
\mathcal{L}_\ve v_\ve-\lambda v_\ve \geq \bigl(m \kappa^2-\kappa
(M_1+\ve M_2)+ (c(x,x/\ve^\alpha)-\lambda)\bigr) e^{\kappa W(x)/\ve}
-(c(x,x/\ve^\alpha)-\lambda)>0\quad \text{in}\ \Omega^\prime
\end{equation*}
when $\kappa>\kappa_1:= (M_1+M_2)/m$. Here $\Omega^\prime=\{x\in \Omega;\,
|\nabla W|\geq 1\}$, $M_1=\max
\bigl|b^{i}(x,y)\frac{\partial W}{\partial x_i}(x)\bigr|$, $M_2= \max
\bigl|a^{ij}(x,y)\frac{\partial^2W}{\partial x_i\partial x_j}(x)\bigr|$. On
the other hand $\delta:= \inf \{W(x); \, x\in
\Omega\setminus \Omega^\prime\}>0$ therefore $e^{\kappa W(x)/\ve}>2$ in $\Omega\setminus \Omega^\prime$, when
$\kappa>\kappa_2:=(\log 2)/\delta$. Assuming additionally that $\min c(x,y)-\lambda>2\kappa (M_1+M_2)$, we have
\begin{equation*}
\mathcal{L}_\ve v_\ve-\lambda v_\ve \geq (-\kappa
(M_1+\ve M_2)+(c(x,x/\ve^\alpha)-\lambda)) \exp(\kappa
\gamma/\ve)-(c(x,x/\ve^\alpha)-\lambda)>0\qquad \text{in}\
\Omega\setminus \Omega^\prime.
\end{equation*}
 Thus  setting
$\kappa:=\max\{\kappa_1,\kappa_2\}$ and $\Lambda:=2\kappa (M_1+ M_2)-\min
c(x,y)$ we get (\ref{testLowerB_eigenval}).

Now note that $\lambda_\ve$ is also the first eigenvalue (the
eigenvalue with the maximal real part) of the adjoint operator
$
\mathcal{L}^{*}_\ve u=\ve^2\frac{\partial^2}{\partial x_i\partial x_j}
(a^{ij}u) -\ve\frac{\partial}{\partial x_i} (b^i u)+cu,
$
and the corresponding eigenfunction $u^{*}_\ve$ can be chosen
positive in  $\Omega$.  Therefore, if $\lambda_\ve<-\Lambda$ then
$(\mathcal{L}_\ve v_\ve -\lambda_\ve v_\ve) u^{*}_\ve>0$ in $\Omega$ that
contradicts the Fredholm theorem. \end{proof}

The following two results show that, up to extracting a
subsequence, functions $W_\ve$ converge uniformly on compacts in
$\Omega$. For brevity introduce the notation
$$
 d(x)={\rm dist}(x,\partial \Omega).
$$

\begin{lem}
\label{LemLipsch} For every $\kappa>0$ there is a constant
$C_\kappa$,  independent of $\ve$, such that
\begin{equation}
\label{lipsch} |W_\ve(x)-W_\ve(z)|\leq C_\kappa (|x-z|+\ve) \qquad
\text{when}\ x,z\in\Omega, \ \min\{d(x),d(z)\} \geq
\kappa\ve.
\end{equation}
\end{lem}

\begin{proof} Let $x_0\in \Omega$, $d(x_0)\geq \kappa \ve$.
Changing the variables $x\mapsto x_0+\ve y$ in operator (\ref{operator}) we find that $v_\ve=u_\ve(x_0+\ve y)$ satisfies the equation
$$
a^{ij}_\ve(y)
\frac{\partial^2 v_\ve}{\partial y_i\partial y_j}+b^j_\ve(y)\frac{\partial
v_\ve}{\partial y_j} +(c_\ve(y)-\lambda_\ve) v_\ve=0\qquad
\text{when}\ |y|<\kappa.
$$
Note that
$a^{ij}_\ve$, $b^{j}_\ve$ and $c_\ve$ are uniformly bounded and
$a^{ij}_\ve\zeta^i\zeta^j\geq m |\zeta|^2>0$ $\forall \zeta\not=0$.
Therefore by Harnack's inequality (see, e.g., \cite{KS}) we have
$v_\ve(y)\leq
C^\prime_{\kappa} v_\ve(y^\prime)$ when $|y|,|y^\prime|<\kappa/2$, where
$C^\prime_\kappa>0$ is independent of $x_0$ and $\ve$. Thus
\begin{equation}
\label{harnack} |W_\ve(x)-W_\ve(z)|\leq(\log C^\prime_{\kappa})\, \ve, \qquad
\text{if} \ |x-z|<\kappa \ve/2 \quad \text{and}\
\min\{d(x), d(z)\} \geq
\kappa\ve.
\end{equation}
Since $\partial\Omega$ is $C^2$-smooth and $\Omega$ is connected,
every two points $x,\,z\in\Omega$ such that $d(x)\geq \kappa\ve$, $d(z)\geq \kappa\ve$ can be connected by a
chain of segments $[x_i,x_{i+1}]$, $i=\overline {1,n_\ve}$ with
$|x_{i+1}-x_i|<\kappa \ve/2$, ${\rm dist}(x_i,\partial\Omega)\geq
\kappa\ve$  and $n_\ve\leq C^{\prime\prime}_\kappa (|x-z|/\ve+1)$.
(We assume that $\ve$ is sufficiently small.) Then iterating
(\ref{harnack}) we obtain (\ref{lipsch}).
\end{proof}


\begin{lem} \label{boundarymax}
Let $\phi(x)\in C^2(\overline{\Omega})$ then
every maximum point $x_\ve$ of the function
$u_\ve(x) e^{\phi(x)/\ve}$ satisfies
$
d(x_\ve)\geq \kappa \ve
$
with some $\kappa>0$ independent of $\ve$.
\end{lem}

\begin{proof}
Consider the function $v_\ve(x)=u_\ve e^{(\phi(x)-\rho_\ve(x))/\ve}$, where
$\rho_\ve=2d(x)-\beta d^2(x)/\ve$ and $\beta$ is a positive
parameter to be chosen later. Since
$\mathcal{L}_\ve u_\ve-\lambda_\ve u_\ve=0$ we have
$$
0={\mathcal L}_\ve (e^{(-\phi(x)+\rho_\ve(x))/\ve}v_\ve)
-\lambda_\ve e^{(-\phi(x)+\rho_\ve(x))/\ve}v_\ve
=e^{(-\phi(x)+\rho_\ve(x))/\ve}\tilde {\mathcal L}_\ve v_\ve,
$$
the operator $\tilde {\mathcal L}_\ve$ being given by
\begin{equation*}
\tilde {\mathcal L}_\ve v=\ve^2  a^{ij}(x,x/\ve^\alpha)\frac{\partial^2 v}
{\partial x_i\partial x_j} +
\ve\bigl(b_j(x,x/\ve^\alpha)-2a^{ij}\frac{\partial}{\partial x_i} (\phi-\rho_\ve)\bigr)
\frac{\partial v}{\partial x_j}
+\tilde c_\ve(x)  v,
\end{equation*}
with
\begin{equation*}
\tilde
c_\ve(x)=c(x,x/\ve^\alpha)-\lambda_\ve+
H(\nabla(\phi-\rho_\ve),x,x/\ve^\alpha)
-\ve a^{ij}(x,x/\ve^\alpha)\frac{\partial^2}{\partial x_i\partial x_j}
\bigl(\phi(x)-\rho_\ve(x)\bigr).
\end{equation*}
The coefficient $\tilde c_\ve(x)$ depends on the parameter $\beta$
through the derivatives of $\rho_\ve$, and we  can choose $\beta$ so
large  that $\tilde c_\ve<0\ \text{in}\
\Omega\setminus\Omega_{\ve/\beta}$ for every $\ve>0$, where
$\Omega_{\ve/\beta}=\{x\in\Omega;\,d(x)>\ve/\beta\}$. Indeed,
$|\nabla \rho_\ve| \leq 4$ when  $d(x)\leq\ve/\beta $ while
$\lambda_\ve\geq -\Lambda$ (cf. Lemma \ref{boundsforlambda}) and
\begin{equation*}
a^{ij}(x,x/\ve^\alpha)\frac{\partial^2 \rho_\ve} {\partial
x_i\partial x_j} \leq -\frac{2}{\ve}\beta
a^{ij}(x,x/\ve^\alpha)\frac {\partial d(x)}{\partial x_i} \frac
{\partial d(x)}{\partial x_j}+C\leq -2m \beta/\ve +C \qquad \text{in}\
\Omega\setminus\Omega_{\ve/\beta},
\end{equation*}
where $C$ is independent of $\ve$ and $\beta$. Thus there is
$\beta>0$ such that $\tilde c_\ve(x)<0$ in
$\Omega\setminus\Omega_{\ve/\beta}$. With this choice of
$\beta$ we get by the maximum principle applied to the equation
$\tilde {\mathcal L}_\ve v_\ve=0$, that $ v_\ve \leq \max\{v_\ve(x); \
x\in\partial\Omega_{\ve/\beta}\}$ in
$\Omega\setminus\Omega_{\ve/\beta}$, or $u_\ve e^{\phi(x)/\ve}
\leq e^{(\rho_\ve(x)-\ve/\beta)/\ve} \max\{u_\ve(x)
e^{\phi(x)/\ve}; \ x\in\partial \Omega_{\ve/\beta}\}$ when
$d(x)<\ve/\beta$. Since $\rho_\ve(x)<\ve/\beta$ when
$d(x)<\ve/\beta$, setting $\kappa=1/\beta$ we get the desired
bound.
\end{proof}

\begin{cor} \label{boundedness}
There is a constant $\kappa>0$ such that every maximum point
$x_\ve$ of $u_\ve$ satisfies $d(x_\ve)\geq \kappa\ve$.
\end{cor}

\begin{proof}
We simply apply Lemma \ref{boundarymax} with $\phi\equiv 0$.
\end{proof}

Lemma \ref{LemLipsch}  and Corollary \ref{boundedness} imply that
we can extract a subsequence of functions $W_\ve$ converging
uniformly on every compact in $\Omega$ to a limit $W(x)\in
C(\overline{\Omega})$. Moreover, by Lemma \ref{LemLipsch} the
function $W$ is Lipschitz continuous and $\max\{
|W_\ve(x)-W(x)|;\, x\in\Omega,\, d(x)\geq\kappa\ve\}\to 0$ for
every $\kappa>0$.

\begin{cor} \label{boundforminima}
Let $W_\ve$ converge (along a subsequence)  to a function $W$
uniformly on every compact in $\Omega$. Then, for every $\phi\in
C^2(\overline{\Omega})$, we have
\begin{equation*}
\lim_{\ve\to 0}\min\{W_\ve(x)-\phi(x); x\in\overline{\Omega}\}
=\min\{W(x)-\phi(x); x\in\overline{\Omega}\}.
\end{equation*}
\end{cor}

\begin{proof}  We know that $\max\{ |W_\ve(x)-W(x)|;\, x\in\Omega,\, d(x)\geq \kappa\ve\}\to 0$ as
$\ve\to 0$ for every $\kappa>0$. On the other hand, by Lemma \ref{boundarymax},
$\min\{W_\ve(x)-\phi(x); x\in\overline{\Omega}\}$ is attained at a
point located on the distance at least $\kappa_0\ve$ from the
boundary $\partial\Omega$ for some $\kappa_0>0$ independent of
$\ve$. This implies that
$
\liminf_{\ve\to 0}\min\{W_\ve(x)-\phi(x); x\in\overline{\Omega}\}
\geq\min\{W(x)-\phi(x); x\in\overline{\Omega}\}.
$
The opposite inequality
$
\limsup_{\ve\to 0}\min\{W_\ve(x)-\phi(x); x\in\overline{\Omega}\}
\leq\min\{W(x)-\phi(x); x\in\overline{\Omega}\}
$
is an easy consequence of the uniform convergence of $W_\ve$ to
$W$ on compacts in $\Omega$.
\end{proof}

\section{Vanishing viscosity limit}

This section is devoted to the proof of Theorem \ref{mainth}.
According to the results of the previous section
we can assume that, up to a
subsequence,
\begin{equation}
\label{convlambda} \lambda_\ve\to \lambda
\end{equation}
and
\begin{equation}
\label{convW} \max\{|W_\ve(x) -W(x)|;\, x\in\Omega,\, d(x)\geq
{\kappa\ve}\}\to 0 \ \forall \kappa>0,
\end{equation}
where $W$ is some Lipschitz continuous function. We are going to
show that the pair $\lambda$ and $W$ is a solution of problem
(\ref{homprob}), (\ref{hombcond}).

We  follow
the same scheme for $\alpha>1$, $\alpha=1$ and $\alpha<1$.

For an arbitrary function $\phi\in C^2(\overline{\Omega})$, let
$W-\phi$ attain strict minimum at a point $x_0\in
\overline\Omega$. We construct test functions of the form
$$
\phi_\ve=\tilde\phi_\ve-|x-x_0|^2,
$$
where $\tilde\phi_\ve\to\phi$ uniformly in $\overline{\Omega}$. Moreover, we will chose
functions $\tilde\phi_\ve$ with uniformly bounded first derivatives and such that
\begin{equation}
\label{testconv}
-\ve a^{ij}(x_\ve,x_\ve/\ve^\alpha)
\frac{\partial^2 \tilde\phi_\ve}{\partial x_i\partial x_j}(x_\ve)+
H\bigl(\nabla \tilde\phi_\ve(x_\ve), x_\ve,x_\ve/\ve^\alpha\bigr)\to
\overline H(\nabla \phi(x_0),x_0)
\end{equation}
for every sequence of points $x_\ve\in\Omega$ such that
$x_\ve\to x_0$. The existence of such functions $\tilde\phi_\ve$ will be justified later on.

Let us now show that  minima points of functions $W_\ve-\phi_\ve$ converge to $x_0$.
Indeed,
$$|\min (W_\ve-\phi_\ve)-\min (W_\ve-\phi+|x-x_0|^2)|\to 0$$
since $\phi_\ve(x)\to\phi(x)-|x-x_0|^2$ uniformly in
$\overline{\Omega}$.  If $x_\ve$ is a minimum point of
$W_\ve-\phi_\ve$ then in view of Corollary \ref{boundforminima} we
have
$$
\min (W_\ve-\phi_\ve)=W_\ve(x_\ve)-\phi(x_\ve)
+|x_\ve-x_0|^2+\bar{o}(1)\geq \min
(W-\phi)+|x_\ve-x_0|^2+\bar{o}(1).
$$
On the other hand, choosing a sequence of points $\tilde
x_\ve\in\Omega$ such that $d(\tilde x_\ve)=|\tilde x_\ve-x_0|=\ve$  we have
$$
\min (W_\ve-\phi_\ve)\leq W_\ve(\tilde
x_\ve)-\phi_\ve(\tilde x_\ve)\to W(x_0)-\phi(x_0)=\min
(W_0-\phi),$$
therefore $x_\ve\to x_0$.

Now, if $x_\ve$ is a minimum point of $(W_\ve-\phi_\ve)$ then
clearly $x_\ve\in \Omega$, hence $\nabla \phi_\ve(x_\ve)=\nabla
W_\ve(x_\ve)$ and $a^{ij}(x_\ve,x_\ve/\ve^\alpha) \frac{\partial^2
W_\ve}{\partial x_i\partial x_j}(x_\ve) \geq
a^{ij}(x_\ve,x_\ve/\ve^\alpha) \frac{\partial^2\phi_\ve}{\partial
x_i\partial x_j}(x_\ve). $ We use (\ref{eqforW}) to get
$$
H\bigl(\nabla W_\ve(x_\ve),x_\ve,x_\ve/\ve^\alpha\bigr)-\lambda_\ve\geq \ve a^{ij}(x_\ve,x_\ve/\ve^\alpha)
\frac{\partial^2\phi_\ve}{\partial x_i\partial x_j}(x_\ve),
$$
or
\begin{equation}
\label{prelimit}
-\ve a^{ij}(x_\ve,x_\ve/\ve^\alpha)
\frac{\partial^2\tilde\phi_\ve}{\partial x_i\partial x_j}(x_\ve)+
H\bigl(\nabla \tilde \phi_\ve(x_\ve),x_\ve,x_\ve/\ve^\alpha\bigr)-\lambda_\ve+\underline{O}(\ve+|x_\ve-x_0|)\geq 0.
\end{equation}
Next we use
(\ref{testconv}) to pass to the limit in (\ref{prelimit}) as
$\ve\to 0$ that leads to the desired inequality
\begin{equation}
\label{final1} \overline H(\nabla \phi(x_0), x_0)\geq \lambda.
\end{equation}

If  $W-\phi$ attains strict maximum at a point $x_0\in\Omega$,  we
argue similarly. We construct test functions of the form
$\phi_\ve=\tilde\phi_\ve$ with $\tilde\phi_\ve$ satisfying
(\ref{testconv}) and such that $\tilde\phi_\ve\to\phi$ uniformly
in $\overline{\Omega}$. Note that this time
$x_0\not\in\partial\Omega$ and therefore we can always chose a
sequence of local maxima points $x_\ve$ of $W_\ve-\phi_\ve$ converging to $x_0$. Then
using the same arguments as in the proof of (\ref{final1}) we derive
\begin{equation}
\label{final2} \overline H(\nabla \phi(x_0), x_0)\geq \lambda.
\end{equation}
Thus $W(x)$ is a viscosity solution of (\ref{homprob}),
(\ref{hombcond}).

It remains to construct functions $\tilde\phi_\ve$ that have
uniformly bounded first derivatives, satisfy (\ref{testconv}) and
converge to $\phi$ uniformly in $\overline{\Omega}$.


\medskip\noindent
{\bf Case  $\alpha>1$.} We set
$$
\tilde\phi_\ve(x)=\phi(x)+\ve^{2\alpha-1}\theta(x/\ve^\alpha),
$$
where $\theta(y)$ is a $Y$-periodic solution of
\begin{equation}
\label{test1}
-a^{ij}(x_0,y)\frac{\partial^2\theta }{\partial y_i\partial y_j} =
\overline H(\nabla \phi(x_0),x_0)
-H(\nabla\phi(x_0),x_0,y).
\end{equation}
Thanks to (\ref{effH1}) such a solution does
exist, for (\ref{effH1}) is nothing but the solvability condition for (\ref{test1}).
Moreover, since the coefficients and the right hand side in
(\ref{test1}) are Lipschitz continuous, $\theta\in C^{2,1}$
(see, e.g., \cite{GT}). Therefore if $x_\ve\to x_0$ as $\ve\to 0$, then
we have
\begin{multline*}
-\ve a^{ij}(x_\ve,x_\ve/\ve^\alpha)\frac{\partial^2 \tilde\phi_\ve(x_\ve)}{\partial x_i \partial x_j}+
H\bigl(\nabla\tilde \phi_\ve(x_\ve),x_\ve,x_\ve/\ve^\alpha\bigr)=
-a^{ij}(x_\ve,x_\ve/\ve^\alpha)\frac{\partial^2 \theta(x_\ve/\ve^\alpha)}
{\partial y_i\partial y_j}+\underline{O}(\ve)
\\+H\bigl(\nabla\phi(x_\ve)+\underline{O}(\ve^{\alpha-1}),x_\ve,x_\ve/\ve^\alpha\bigr)=
-a^{ij}(x_0,x_\ve/\ve^\alpha)\frac{\partial^2 \theta(x_\ve/\ve^\alpha)}{\partial y_i\partial y_j}\\
+H\bigr(\nabla\phi(x_0),x_0,x_\ve/\ve^\alpha\bigr)+\underline{O}(|x-x_\ve|+\ve+\ve^{\alpha-1})
=\overline H(\nabla \phi(x_0),x_0)+\overline{o}(1).
\end{multline*}

\medskip\noindent
{\bf Case $\alpha=1$.} Set $\tilde\phi_\ve(x)=\phi(x)+\ve\theta(x/\ve)$,
where $\theta(y)=-\log\vartheta(y)$ and $\vartheta(y)$ is the
unique (up to multiplication by a positive constant) $Y-$periodic positive solution of
\begin{equation}
\label{test2}
a^{ij}(x_0,y)\frac{\partial^2\vartheta}{\partial y_i\partial y_j}+
\hat b^{j}(y)\frac{\partial \vartheta}{\partial y_j}+
\hat c(y))\vartheta=\overline H(p,x)\vartheta,
\end{equation}
where $p=\nabla \phi(x_0)$, $\hat b^j(y)=b^{j}(x_0,y)-2a^{ij}(x_0,y) p_i$,
$\hat c(y)=a^{ij}(x_0,y) p_ip_j-
b^{j}(x_0,y)p_j+c(x_0,y)$. By a standard elliptic regularity result
we have $\theta\in C^{2,1}$
(see \cite{GT}).

Simple calculations
show that $\theta(y)$ satisfies
\begin{equation}
-a^{ij}(x_0,y)\frac{\partial^2 \theta}{\partial y_i\partial y_j}+
H\bigl(\nabla\phi(x_0)+\nabla\theta,x_0,y\bigr)=\overline H(\nabla\phi(x_0),x_0).
\end{equation}
Then we easily conclude that
\begin{multline*}
-\ve a^{ij}(x_\ve,x_\ve/\ve)\frac{\partial^2 \tilde\phi_\ve(x_\ve)}{\partial x_i \partial x_j}+
H\bigr(\nabla\tilde \phi_\ve(x_\ve),x_\ve,x_\ve/\ve\bigr)=
-a^{ij}(x_\ve,x_\ve/\ve)\frac{\partial^2 \theta(x_\ve/\ve)}{\partial y_i\partial y_j}+\underline{O}(\ve)
\\+H\bigl(\nabla\phi(x_\ve)+\nabla \theta(x_\ve/\ve),x_\ve,x_\ve/\ve\bigr)
=-a^{ij}(x_0,x_\ve/\ve)\frac{\partial^2 \theta(x_\ve/\ve)}{\partial y_i\partial y_j}\\
+H\bigl(\nabla\phi(x_0)+\nabla
\theta(x_\ve/\ve),x_0,x_\ve/\ve\bigr)+\underline{O}(|x_\ve-x_0|+\ve)=
\overline H(\nabla \phi(x_0),x_0)+\overline{o}(1),
\end{multline*}
as soon as $x_\ve\to x_0$ when $\ve\to 0$.

\medskip\noindent
{\bf Case $\alpha<1$.} Set
$\tilde\phi_\ve(x)=\phi(x)+\ve^\alpha\theta_\ve(x/\ve^\alpha)$
where $\theta_\ve$ is a $Y$-periodic solution of the equation
\begin{equation}
\label{test3}
-\ve^{1-\alpha} a^{ij}(x_0,y)\frac{ \partial^2\theta_\ve}{\partial y_i\partial y_j}+
H\bigl(p+ \nabla \theta_\ve(y), x_0,y\bigr)=\overline H_\ve(p,x_0) \
\text{with}\  p=\nabla\phi(x_0) .
\end{equation}
Such a solution exists if $\overline H_\ve(p,x_0)$ coincides the
first eigenvalue $\mu_\ve$
(eigenvalue with the maximal real part)  of the spectral problem
$$
\begin{array}{l}
\displaystyle\ve^{2(1-\alpha)}a^{ij}(x_0,y)\frac{ \partial^2\vartheta_\ve}{\partial y_i\partial y_j}
+\ve^{1-\alpha}\hat b^j(y)\frac{ \partial\vartheta_\ve}{\partial y_j}
+\hat c(y)\vartheta_\ve=
\mu_\ve\vartheta_\ve,
\\
\vartheta_\ve\quad \text{is}\ Y\text{-periodic},
\end{array}
$$
where $\hat b^j$, $\hat c$ are as in (\ref{test2}).
According to the Krein-Rutman theorem $\mu_\ve$ is
a real and simple eigenvalue and the corresponding
eigenfunction  $\vartheta_\ve$ can be
chosen positive. Then a solution of  (\ref{test3}) is given by
$\theta_\ve=-\ve^{1-\alpha}\log \vartheta_\ve$. We invoke now the results obtained in
Section \ref{periodS},
\begin{equation}
\label{MuConv} \overline H_\ve(p,x_0)\to \overline H(p,x_0)=\overline
H(\nabla\phi(x_0),x_0)
\end{equation}
(where the limit $\overline H(p,x_0)$ is described in
(\ref{effH2})),
\begin{equation}
\label{MuBound}
\|\partial^2 \vartheta_\ve/\partial y_i\partial y_j\|_{L^\infty}\leq
C/\ve^{1-\alpha}
\end{equation}
This allows us to show (\ref{testconv}) similarly to other cases considered above,
\begin{multline*}
-\ve a^{ij}(x_\ve,x_\ve/\ve^\alpha)\frac{\partial^2 \tilde\phi_\ve(x_\ve)}{\partial x_i \partial x_j}+
H\bigl(\nabla\tilde \phi_\ve(x_\ve),x_\ve,x_\ve/\ve^\alpha\bigr)=
-\ve^{1-\alpha}a^{ij}(x_\ve,x_\ve/\ve^\alpha)\frac{\partial^2 \theta_\ve(x_\ve/\ve^\alpha)}{\partial y_i\partial y_j}+\underline{O}(\ve)
\\+H\bigl(\nabla\phi(x_\ve)+\nabla \theta_\ve(x_\ve/\ve^\alpha),x_\ve,x_\ve/\ve^\alpha\bigr)
=-\ve^{1-\alpha}a^{ij}(x_0,x_\ve/\ve^\alpha)\frac{\partial^2 \theta_\ve(x_\ve/\ve^\alpha)}{\partial y_i\partial y_j}\\
+H\bigl(\nabla\phi(x_0)+\nabla \theta_\ve(x_\ve/\ve^\alpha),x_0,x_\ve/\ve^\alpha\bigr)+\underline{O}(|x-x_\ve|+\ve)=
\overline H(\nabla \phi(x_0),x_0)+\overline{o}(1).
\end{multline*}
Theorem \ref{mainth} is completely proved. \hfill $\square$

\section{Lower bound for eigenvalues via blow up analysis}
\label{SecLowerbound}

From now on we consider in details the special case when
$\alpha=1$ and $c(x,y)=0$. Under the assumption
(\ref{fixedpoints}) the eigenvalues $\lambda_\ve$ of
(\ref{operator}) converge to zero as $\ve\to 0$. We are interested
in the more precise (up to the order $\ve$) asymptotics for
$\lambda_\ve$. 
We resolve this
question by local analysis near points $\xi$ of the Aubry set
$\mathcal{A}_{\overline H}$ of the effective Hamiltonian.

Fix a point $\xi \in\mathcal{A}_{\overline H}$. 

Applying the maximum principle we see that $\lambda_\ve<0$. On the
other hand, it is well known 
that the eigenvalue $\lambda_\ve$ is
given by
\begin{equation}
\lambda_\ve=\inf\Bigl\{\sup_{x\in\Omega}\frac{\mathcal{L}_\ve \phi}{\phi}; \ \phi\in C^2(\Omega)\cap C(\overline\Omega),\ \phi>0 \
\text{in}\ \Omega,\,\phi=0\ \text{on}\
\partial \Omega\Bigr\}. \label{lambdaformula}
\end{equation}
Therefore, for every given $\delta>0$, we have $$ \lambda_\ve\geq
\ve \tilde\sigma_\ve,
$$
where $\tilde\sigma_\ve<0$ is the first eigenvalue of the equation
$\mathcal{L}_\ve v_\ve-\delta|x-\xi|^2 v_\ve=\ve\tilde \sigma_\ve v_\ve$
or
\begin{equation}
\ve a^{ij}(x,x/\ve)\frac{\partial^2 v_\ve} {\partial x_i\partial
x_j} + b^{j}(x,x/\ve)\frac{\partial v_\ve}{\partial
x_j}-\frac{\delta|x-\xi|^2}{\ve}v_\ve =\tilde\sigma_\ve v_\ve \qquad
\text{in}\ \Omega\label{Singpertproblemtest}
\end{equation}
with the Dirichlet condition
$v_\ve=0$  on $\partial\Omega$.
We assume hereafter that the first eigenfunction $v_\ve$ is normalized by $v_\ve(\xi)=1$.

Let us transform (\ref{Singpertproblemtest})
to a form more convenient for the analysis.
First, after
changing variables $z=(x-\xi)/\sqrt{\ve}$ and setting
$w_\ve(z)=v_\ve(\xi+\sqrt{\ve}z)$
equation (\ref{Singpertproblemtest}) becomes
\begin{equation}
a^{ij}_{\xi,\xi/\ve}\bigl(\sqrt{\ve} z,z/\sqrt{\ve}\bigr)\frac{\partial^2 w_\ve} {\partial
z_i z_j}+ \frac{b^{j}_{\xi,\xi/\ve}\bigl(\sqrt{\ve} z,z/\sqrt{\ve}\bigr)}
{\sqrt{\ve}}\frac{\partial w_\ve}{\partial z_j}-\delta |z|^2w_\ve
=\tilde\sigma_\ve w_\ve \quad \text{in}\
(\Omega-\xi)/\sqrt{\ve}.\label{Singpertproblemtest1}
\end{equation}
Here (and below) the subscript "$\xi,\xi/\ve$" denotes
the shift (translation) by $\xi$ in $x$ and by $\xi/\ve$ in $y$,
i.e., for instance, $a^{ij}_{\xi,\xi/\ve}(x,y)=a^{ij}_{\xi,\xi/\ve}(\xi+x,\xi/\ve+y)$.
Next multiply (\ref{Singpertproblemtest1}) by
$\theta^\ast_{\xi,\xi/\ve}\bigl(\sqrt{\ve} z,z/\sqrt{\ve}\bigr)$,
$\theta^\ast(x,y)$ being given by (\ref{effdualproblresonans}), to
find after simple rearrengements
\begin{multline}
\frac{\partial}{\partial
z_i}\Bigl(\theta^\ast_{\xi,\xi/\ve}\bigl(\sqrt{\ve}z,z/\sqrt{\ve}\bigr)a^{ij}_{\xi,\xi/\ve}\bigl(\sqrt{\ve}z,z/\sqrt{\ve}\bigr)\frac{\partial
w_\ve} {\partial
z_j}\Bigr)+\frac{S^j_{\xi,\xi/\ve}\bigl(\sqrt{\ve}z,z/\sqrt{\ve}\bigr)}{\sqrt{\ve}}\frac{\partial
w_\ve}{\partial z_j} \\+ \Bigl(\frac{\overline{b}^j(\sqrt{\ve}z+\xi)}{\sqrt{\ve}}+\sqrt{\ve} h^j_\ve(z)\Bigr)\frac{\partial
w_\ve}{\partial z_j}=(\tilde\sigma_\ve+\delta |z|^2) \theta^\ast_{\xi,\xi/\ve}\bigl(\sqrt{\ve}z,z/\sqrt{\ve}\bigr)w_\ve,\label{firsttwoterms}
\end{multline}
where $S^j_{\xi,\xi/\ve}(x,y)$ is obtained by shifts (as described above) from
$$
S^j(x,y)=b^j(x,y)\theta^\ast(x,y)-\frac{\partial}{\partial y_i}
\Bigl(a^{ij}(x,y)\theta^\ast(x,y)\Bigr)-\overline {b}^j(x),
$$
and $h_\ve^j$ are uniformly bounded functions.
Since $\theta^\ast$ solves (\ref{effdualproblresonans}),  the $Y$-periodic vector field
$S(x,y)=(S^1(x,y),\dots,S^N(x,y))$ is divergence free, for every
fixed $x$, and (due to the definition of $\overline b$)
this field has zero mean over the period. Therefore we can find the representation
$$
S^j(x,y)=\frac{\partial}{\partial y_i}T^{ij}(x,y)\ \text{with $Y$-periodic in $y$ skew-symmetric}\ T^{ij}(x,y)\   (T^{ij}=-T^{ji}).
$$
Moreover, functions $T^{ij}$ are (can be chosen) continuous with bounded
derivatives $\partial T^{ij}/\partial x_k$. We can thus  rewrite (\ref{firsttwoterms}) as
\begin{equation}
\frac{\partial}{\partial
z_i}\Bigl(q^{ij}_{\xi,\xi/\ve}\bigl(\sqrt{\ve}z,z/\sqrt{\ve}\bigr)
\frac{\partial w_\ve}
{\partial z_j}\Bigr)+\Bigl(\frac{\overline{b}^j(\sqrt{\ve}z+\xi)}{\sqrt{\ve}}+\sqrt{\ve}\tilde h^j_\ve(z)\Bigr)\frac{\partial
w_\ve}{\partial z_j}=(\tilde\sigma_\ve+\delta |z|^2) \theta^\ast_{\xi,\xi/\ve}\bigl(\sqrt{\ve}z,z/\sqrt{\ve}\bigr)w_\ve,
\label{Singpertproblemtest3}
\end{equation}
where $q^{ij}_{\xi,\xi/\ve}(x,y)=q^{ij}(\xi +x,\xi/\ve+y)$,
$q^{ij}(x,y)=\theta^\ast(x,y)a^{ij}(x,y)+T^{ij}(x,y)$, and $\tilde h^j_\ve$ are uniformly bounded functions.
Note that on every fixed
compact we have
$$
\frac{\overline{b}^j(\sqrt{\ve}z+\xi)}{\sqrt{\ve}}\to z_i \frac{\partial \overline{b}^j}{\partial x_i}(\xi)
$$
uniformly in $z$ as $\ve\to 0$.

\begin{lem}
\label{uniformpreciseboundeigenvalue} If $\overline b(\xi)=0$ for
some $\xi\in\overline\Omega$ then the first eigenvalue
$\lambda_\ve$ of the operator (\ref{operator-redu}) satisfies the
bound $-\Lambda \ve\leq\lambda_\ve< 0$
with some $\Lambda>0$ independent of $\ve$.
\end{lem}

\begin{proof}
We know that $\lambda_\ve<0$ and in the proof
of the lower bound we assume first $\xi\in\Omega$. Then
(\ref{Singpertproblemtest3}) holds in $B_2=\{z;\, |z|<2\}$ for
sufficiently small $\ve$. Let us write
(\ref{Singpertproblemtest3})
in the operator form $\mathcal{L}^{(aux)}_\ve w_\ve
=\tilde\sigma_\ve\theta^\ast_{\xi,\xi/\ve}\bigl(\sqrt{\ve}z,z/\sqrt{\ve}\bigr)w_\ve$ and consider
the parabolic equation for the operator $\mathcal{L}^{(aux)}_\ve$
\begin{equation*}
\frac{\partial \tilde w_\ve}{\partial t}-\mathcal{L}^{(aux)}_\ve \tilde w_\ve=0 \qquad  \text{in }\ (0,+\infty)\times
B_2,
\end{equation*}
subject to the initial condition $\tilde w_\ve(0,z)=w_\ve(z)$ and the boundary condition $\tilde w_\ve(t,z)=0$ on
$(0,+\infty)\times\partial\Omega$. The solution $\tilde w_\ve$ of this problem has the pointwise bound
$\tilde w_\ve(t,z)\leq \exp\bigl(\tilde\sigma_\ve(\min \theta^\ast)t\bigr)\, w_\ve(z)$. This follows by the maximum
principle applied to
$$
\Bigl(\frac{\partial }{\partial t}-\mathcal{L}^{(aux)}_\ve\Bigr)
\Bigl(\exp\bigl(\tilde\sigma_\ve(\min \theta^\ast)t\bigr)\, w_\ve(z)-\tilde w_\ve(t,z)\Bigr)=\tilde \sigma_\ve
\Bigl(\min \theta^*-\theta^\ast_{\xi,\xi/\ve}\bigl(\sqrt{\ve}z,z/\sqrt{\ve}\bigr)\Bigr)w_\ve(z)\geq 0.
$$
On the other hand, since the coefficients of the operator $\mathcal{L}_\ve$ are uniformly bounded
and the uniform ellipticity bound
$\theta^\ast_{\xi,\xi/\ve}
\bigl(\sqrt{\ve}z,z/\sqrt{\ve}\bigr)a^{ij}_{\xi,\xi/\ve}\bigl(\sqrt{\ve}z,z/\sqrt{\ve}\bigr) \zeta^i\zeta^j
\geq (\min \theta^\ast) m \, |\zeta|^2$ holds, by the Aronson estimate (see \cite{A}) we have
$$ \min\{\tilde
w_\ve(1,z);\ z\in B_1\}\geq M \min\{\tilde w_\ve(0,z);\ z\in B_1\}
$$
with $M>0$ independent of $\ve$, where $B_1$ is the unit ball $B_1=\{z;\, |z|<1\}$.
This yields
$$
e^{(\min \theta^\ast)\tilde\sigma_\ve}\min_{B_1} w_\ve\geq \min\{\tilde w_\ve(1,z);\ z\in
B_1\}\geq M \min\{\tilde w_\ve(0,z);\ z\in B_1\}=M\min_{B_1}w_\ve,
$$
i.e.  $\tilde\sigma_\ve\geq \log M /\min \theta^\ast=:-\Lambda$. Thus $\tilde\sigma_\ve\geq -\Lambda$ and
$\lambda_\ve\geq -\Lambda\ve$.

Finally, in the case $\xi\in \partial\Omega$ we can repeat the
above argument taking $\xi_\ve\in \Omega$ in place of $\xi$, with
$|\xi_\ve-\xi|={\rm dist}(\xi_\ve,\partial\Omega)=2\sqrt\ve$.\end{proof}


In the proof of Lemma \ref{uniformpreciseboundeigenvalue} we have got
a uniform lower bound for $\tilde\sigma_\ve$ which (in conjunction with the obvious inequality  $\tilde\sigma_\ve<0$)
allows one to obtain
uniform bounds for the norm of $w_\ve$ in $C^{0,\beta}(K)$ (with $\beta>0$ depending only on
bounds for coefficients in (\ref{Singpertproblemtest3})) and
$H^1(K)$, for every compact $K$ (see, e.g., \cite{GT}).
Thus, up to extracting a subsequence, $w_\ve\to w$ in $C_{\rm loc}({\mathbb R}^N)$
and $\tilde \sigma_\ve\to\tilde
\sigma$. Moreover, using well established homogenization technique
based on the ${\rm div}$-${\rm curl}$ Lemma we get that $w$ solves
\begin{equation}
Q^{ij}\frac{\partial^2 w} {\partial z_i\partial z_j}+
 B^{ij}z_i\frac{\partial
w} {\partial z_j}-\delta|z|^2w=\tilde\sigma w\qquad \text{in}\
\mathbb{R}^N, \label{homogequationrescaled}
\end{equation}
where $Q^{ij}=Q^{ji}$ are some constant coefficients satisfying
the ellipticity condition (actually, one can check that $Q^{ij}=\frac{1}{2}\frac{\partial^2 \overline{H}}{\partial p_i\partial p_j}(0,\xi)$)
and
$$
B^{ij}(=B^{ij}(\xi))=\frac{\partial \overline b^j}{\partial x_i}(\xi).
$$
Since we assumed the normalization $w_\ve(0)=1$,
we see that $w(z)$ is a nontrivial solution of
(\ref{homogequationrescaled}). Moreover, if $z_\ve$ is a maximum
point of $w_\ve(z)$ we get from (\ref{Singpertproblemtest1})
$|z_\ve|^2\leq -\tilde\sigma_\ve/\delta$ therefore, thanks to
Lemma \ref{uniformpreciseboundeigenvalue}, $|z_\ve|\leq C$. It
follows that $w(z)$ is a bounded positive solution of
(\ref{homogequationrescaled}).
A solution of (\ref{homogequationrescaled}) can be constructed in
the form $w(z)=e^{-\Gamma^{ij}_\delta z_i z_j}$ with a symmetric
positive definite matrix
$(\Gamma^{ij}_\delta)_{i,j=\overline{1,N}}$.
Indeed, consider the following matrix Riccati equation
$$
4\Gamma_\delta Q\Gamma_\delta-B\Gamma_\delta-\Gamma_\delta
B^\ast-\delta I=0,
$$
where $I$ denotes the unit matrix. It is well-known that there
exists a unique positive definite solution $\Gamma_\delta$ (since
$\delta>0$ and $Q$ is positive definite, see \cite{LR}). Then
$w(z)=e^{-\Gamma^{ij}_\delta z_i z_j}$ is a positive bounded
solution of (\ref{homogequationrescaled}) corresponding to the
eigenvalue $\tilde \sigma=-2{\rm tr}(Q\Gamma_\delta)$. Next
observe that by means of the gauge transformation $\tilde
w(z)=e^{-r |z|^2}w(z)$ ($r>0$) equation
(\ref{homogequationrescaled}) is reduced to
\begin{equation*}
Q^{ij}\frac{\partial^2 \tilde w} {\partial z_i\partial z_j}+
 (B^{ij}+4r Q^{ij})z_i\frac{\partial
\tilde w} {\partial z_j}+(4r^2 Q^{ij}z_i z_j+2r\, {\rm tr} Q+2r
B^{ij}z_i z_j-\delta|z|^2)\tilde w=\tilde\sigma \tilde w\qquad
\text{in}\ \mathbb{R}^N.
\end{equation*}
For sufficiently small $r>0$ we have $((4r^2 Q^{ij}z_i z_j+2r
\,{\rm tr} Q+2r B^{ij}z_i z_j-\delta|z|^2)\to -\infty$ and $\tilde
w(z)\to 0$ as $|z|\to \infty$. Then, according to
\cite{PS}, the eigenvalue $\tilde\sigma$ corresponding to such a
positive eigenfunction $\tilde w$ (vanishing as $|z|\to\infty$) is
unique. Thus $\tilde \sigma=-2{\rm tr}(Q\Gamma_\delta)$, and
summarizing the above analysis we have $\liminf_{\ve\to
0}\lambda_\ve/\ve\geq -2{\rm tr}(Q\Gamma_\delta)$. Finally note
that $\Gamma_\delta$ converges to the maximal positive
semi-definite solution of the Bernoulli equation (see, e.g., \cite{LR})
\begin{equation}
4\Gamma Q\Gamma-B\Gamma-\Gamma B^\ast=0, \label{Bernoulli}
\end{equation}
as $\delta\to+0$. Calculations presented in Appendix \ref{C}
show that
$-2{\rm tr}(Q\Gamma)= \sigma(\xi)$ with $\sigma(\xi)$ being the
sum of negative real parts of the eigenvalues of $-B(\xi)$. Thus,
after maximizing in $\xi\in\mathcal{A}_{\overline H}$ we get
\begin{equation}
\label{optimallowerbound}
 \liminf_{\ve\to 0}\lambda_\ve/\ve\geq \overline \sigma=\max \{\sigma(\xi);\, \xi\in \mathcal{A}_{\overline H}\} .
\end{equation}

\section{Upper bound for eigenvalues and selection of the additive eigenfunction}
\label{SecUpperbound}

In this section we derive an upper bound for eigenvalues which
completes the proof of the asymptotic expansion formula
(\ref{selection}). Similarly to the previous section we make use
of the blow up analysis near points of the Aubry set. We consider
here only special  (so-called {\it significant}) points of the Aubry
set, where we can control the asymptotic behavior of rescaled
eigenfunctions at infinity. We will show that only these special
points matter to the leading term of eigenvalues and
eigenfunctions.


Due to Theorem \ref{mainth}, up to extracting a subsequence, the
functions $W_\ve=-\ve\log u_\ve$ converge uniformly on compacts to
a viscosity solution $W$ of $\overline H(\nabla W(x),x)=0$ in $\Omega$,
$\overline H(\nabla W(x),x)\geq 0$ on $\partial \Omega$. It follows that $W$
has the representation $W(x)=\min\{d_{\overline H}(x,\xi)+W(\xi);\,
\xi\in\mathcal{A}_{\overline H}\}$.

We will say that point $\xi\in \mathcal{A}_{\overline H}$ is {\it
significant} if
$$
W(x)=d_{\overline H}(x,\xi)+W(\xi) \qquad \text{in a neighborhood of}\ \xi.
$$
Otherwise we call $\xi$ {\it negligible}. For every negligible point
$\xi\in \mathcal{A}_{\overline H}$ there are sequences $x^k\to\xi$
and $\xi^k\in\mathcal{A}_{\overline H}\setminus\{\xi\}$ such that
$d_{\overline H}(x^k,\xi)+W(\xi)>d_{\overline
H}(x^k,\xi^k)+W(\xi^k)$. Passing to the limit (possibly along a
subsequence) and using the continuity of the distance function we
get $d_{\overline H}(\xi,\xi^\prime)=W(\xi)-W(\xi^\prime)$ for
some $\xi^\prime\in\mathcal{A}_{\overline H}$,
$\xi^\prime\not=\xi$ (we always have $d_{\overline
H}(\xi,\xi^\prime)\geq W(\xi)-W(\xi^\prime)$). Now let us
introduce the (partial) order relation $\preceq$ on
$\mathcal{A}_{\overline H}$ by setting
\begin{equation}
\xi^\prime\preceq \xi\ \Longleftrightarrow\  d_{\overline
H}(\xi,\xi^\prime)=W(\xi)-W(\xi^\prime). \label{orderrelation}
\end{equation}
This relation is clearly reflexive, its transitivity is a consequence of the triangle
inequality $ d_{\overline
H}(\xi,\xi^{\prime\prime})
\leq d_{\overline
H}(\xi,\xi^\prime)+d_{\overline
H}(\xi^\prime,\xi^{\prime\prime})$ while the
antisymmetry follows from the inequality $S_{\overline
H}(\xi,\xi^\prime)>0$ held for all $\xi,\xi^\prime\in
\mathcal{A}_{\overline H}$ with $\xi\not=\xi^\prime$.
Then we see that every minimal element $\xi\in \mathcal{A}_{\overline
H}$ is a significant  point. Since $\mathcal{A}_{\overline H}$ is finite
there exists a minimal element, i.e. there is at least one significant
point $\xi\in \mathcal{A}_{\overline H}$.


Let us fix a significant point $\xi\in\mathcal{A}_{\overline H}$. From
now on we assume that $u_\ve$ is normalized by $u_\ve(\xi)=1$,
unless otherwise is specified; the $W$ will also refer to the limit
of scaled logarithmic transformations of $u_\ve$ normalized in
this way. Thanks to the upper and lower bounds for the eigenvalue
$\lambda_\ve$ we have $\lambda_\ve/\ve\to \sigma$ as $\ve\to 0$
along a subsequence. Then we argue exactly as in the proof of the
lower bound for $\lambda_\ve$. We consider rescaled eigenfunctions
$w_\ve(z)=u_\ve(\xi+\sqrt{\ve} z)$ that are solutions of
\begin{equation*}
\frac{\partial}{\partial
z_i}\Bigl(q^{ij}_{\xi,\frac{\xi}{\ve}}\bigl(\sqrt{\ve}z,z/\sqrt{\ve}\bigr)
\frac{\partial w_\ve}
{\partial z_j}\Bigr)+\Bigl(\frac{\overline{b}^j(\sqrt{\ve}z+\xi)}{\sqrt{\ve}}+\sqrt{\ve}\tilde h^j_\ve(z)\Bigr)\frac{\partial
w_\ve}{\partial z_j}=\frac{\lambda_\ve}{\ve} \theta^\ast_{\xi,\frac{\xi}{\ve}}\bigl(\sqrt{\ve}z,z/\sqrt{\ve}\bigr)w_\ve
\quad
\text{in}\ \frac{\Omega-\xi}{\sqrt{\ve}}.
\end{equation*}
Up to extracting a further subsequence, they converge in $C(K)$ and weakly in
$H^1(K)$  (for every compact $K$) to a positive solution of
\begin{equation}
Q^{ij}\frac{\partial^2 w} {\partial z_i\partial z_j}+
 B^{ij}z_i\frac{\partial
w} {\partial z_j}=\sigma w\qquad \text{in}\ \mathbb{R}^N.
\label{homogequationrescaled1}
\end{equation}
Eigenvalue problem (\ref{homogequationrescaled1}) possesses,
in general, many solutions even in the class of positive
eigenfunctions $w(z)$. We will uniquely identify $\sigma$ studying
the asymptotic behavior of $w(z)$ as $|z|\to \infty$.
More precisely, we will show that 
\begin{equation} \label{specialclass} w(z)e^{\mu|\Pi_s^\ast
z|^2-\nu |\Pi_u^\ast z|^2} \ \text{is bounded on} \ \mathbb{R}^N\
\text{for some} \ \mu>0\ \text{and every}\ \nu>0,
\end{equation}
where
$\Pi_s$ and $\Pi_u$ denote spectral projectors on the invariant
subspaces of the matrix $B$ corresponding to the eigenvalues with
positive and negative real parts (stable and unstable subspaces of
the system $\dot z_i=-B^{ij}z_j$). This allows to use the
following uniqueness result.


\begin{lem}\label{uniquness} Let $w(z)$ be a positive solution of
(\ref{homogequationrescaled1}) satisfying (\ref{specialclass}).
Then $w(z)=Ce^{-\Gamma^{ij}z_i z_j}$ with $C>0$, where $\Gamma$ is
the maximal positive semi-definite solution of (\ref{Bernoulli}).
Moreover, we have $\sigma=-2{\rm tr}(\Gamma Q)$.
\end{lem}
\begin{proof} First observe that
$w(z)=Ce^{-\Gamma^{ij}z_i z_j}$ satisfies (\ref{specialclass}). This
follows from the relation $\Gamma=\Pi_s\Gamma\Pi_s^\ast\geq \gamma \Pi_s\Pi_s^\ast$
where the inequality holds for some $\gamma>0$ in the sense of quadratic forms, see
Proposition \ref{p_app2} in Appendix \ref{C}. It is also clear that $w(z)$ does solve
(\ref{homogequationrescaled1}) with $\sigma= -2{\rm tr}(\Gamma
Q)$.

To justify the uniqueness of $\sigma$ and $w(z)$
we make use of a gauge
transformation $\tilde w(z)=e^{\phi(z)}w(z)$, with a quadratic
function  $\phi(z)$ to be constructed later on, which leads to the equation of the form
\begin{equation}
Q^{ij}\frac{\partial^2 \tilde w} {\partial z_i\partial z_j}+
\tilde B^{ij} z_i\frac{\partial \tilde w} {\partial z_j} +\tilde
C(z)\tilde w =\sigma \tilde w\qquad \text{in}\ \mathbb{R}^N.
\label{homogequaux1}
\end{equation}
We will choose $\phi(z)$ so that $\tilde C(z)\to-\infty$, $\tilde
w(z)\to 0$ as $|z|\to\infty$. Then, by  \cite{PS},
there is a unique $\sigma$ such that
(\ref{homogequaux1}) has a positive solution $\tilde w(z)$
vanishing as $|z|\to \infty$ ($\tilde w(z)$ is also unique up to
multiplication by a positive constant).

Construct $\phi(z)$ by setting
$\phi=r A_s^{ij}z_i z_j-rA_u^{ij}z_i z_j$, with symmetric matrices
$A_s$ and $A_u$, to get $\tilde
B^{ij}=B^{ij}+4rQ^{jl}(A_u^{li}-A_s^{li})$ and
\begin{multline*}
\tilde C(z)=4r^2(A_u^{il}-A_s^{il})Q^{lm}(A_u^{mj}-A_s^{mj})z_i
z_j\\
+r\Bigl(\bigl(B^{il}(A_u^{lj}-A_s^{lj})+(A_u^{il}-A_s^{il})B^{lj}\bigr)z_i
z_j+2{\rm tr}\bigl(Q(A_u-A_s)\bigr)\Bigr).
\end{multline*}
Define $A_s$ and $A_u$ as solutions of the Lyapunov matrix
equations \begin{equation}
 BA_s+A_s B^{\ast}=\Pi_s\Pi_s^\ast, \
BA_u+A_u B^{\ast}=-\Pi_u\Pi_u^\ast, \label{Lyapunovmeq}
\end{equation} given by
\begin{equation}
 A_s=\int_{-\infty}^0 e^{Bt}\Pi_s\Pi_s^\ast e^{B^\ast
t}\,{\rm d}t, \qquad A_u=\int_{0}^\infty e^{Bt}\Pi_u\Pi_u^\ast
e^{B^\ast t}\,{\rm d}t, \label{LyapunovRepr}
\end{equation}
and choose sufficiently small $r_0>0$ in such a way that the matrix
$$
4r(A_u-A_s)Q(A_u-A_s) +B(A_u-A_s)+(A_u-A_s)
B^{\ast}=4r(A_u-A_s)Q(A_u-A_s)-\Pi_s\Pi_s^\ast-\Pi_u\Pi_u^\ast
$$
is negative definite for $0<r<r_0$. Then $\tilde C(z)\to-\infty$ as $|z|\to\infty$.
It remains to see that if $w(z)$ satisfies (\ref{specialclass})
then choosing small enough $r>0$ we have $\tilde w(z)\to 0$ as
$|z|\to\infty$. Here we have used the inequalities $A_s\leq \gamma_1 \Pi_s\Pi_s^*$ and $A_u\geq
\gamma_2 \Pi_u\Pi_u^*$  for some
$\gamma_1,\gamma_2>0$.
\end{proof}

So far we know that $\lambda_\ve/\ve\to \sigma$ and
$w_\ve(z)=u_\ve(\xi+\sqrt{\ve} z)$ converge in
uniformly on compacts
to a positive solution of
(\ref{homogequationrescaled1}). In order to apply Lemma
\ref{uniquness} we need only to show (\ref{specialclass}). To this
end we first construct a quadratic function $\Phi_\mu^\nu(x)$
satisfying
\begin{equation}
\overline{H}(\nabla \Phi_\mu^\nu(x),x)\leq -\delta |x-\xi|^2\qquad \text{in
a neighborhood $U(\xi)$ of} \ \xi \label{subsolution}
\end{equation}
for some $\delta>0$.

\begin{lem}
\label{lemphimunu} Let us set $\phi_s(x):=A^{ij}_{s}x_i x_j$ and
$\phi_u(x):=A^{ij}_{u}x_i x_j$, where $A_{s}$ and $A_u$ are
solutions of the Lyapunov matrix equation (\ref{Lyapunovmeq})
given by (\ref{LyapunovRepr}). Then the function
$$
\Phi_\mu^\nu(x):=\mu\phi_s(x-\xi)-\nu\phi_u(x-\xi)
$$
satisfies (\ref{subsolution}) for some $\delta>0$, provided that
$0<\mu,\nu< r$ and $r>0$ is sufficiently small.
\end{lem}

\begin{proof} We have, as $x\to\xi$,
\begin{multline} \overline H(\nabla \Phi_\mu^\nu(x),x)\leq \overline H(0,x)+
\frac{\partial\overline H }{\partial
p_j}(0,x)\frac{\partial\Phi_\mu^\nu}{\partial x_j}(x)
 +C|\nabla \Phi_\mu^\nu(x)|^2\\=-(x_i-\xi_i)\frac{\partial\overline b^j}{\partial x_i}(\xi)
\frac{\partial\Phi_\mu^\nu}{\partial x_j}(x)
 +C|\nabla \Phi_\mu^\nu(x)|^2+\bar o(|x-\xi|^2)
\\
\leq-2(x_i-\xi_i)B^{ij}(\mu A_s^{jl}-\nu A_u^{jl})(x_l-\xi_l)\\
+C_1(\mu^2|\Pi_s^\ast(x-\xi)|^2+\nu^2|\Pi_u^\ast(x-\xi)|^2)+\bar
o(|x-\xi|^2). \label{expand}
\end{multline}
Note that $B(\mu A_s-\nu A_u)+(\mu A_s-\nu A_u)
B^{\ast}=\mu\Pi_s\Pi_s^\ast+\nu\Pi_u\Pi_u^\ast$, therefore the
first term in the right hand side of (\ref{expand}) can be written
as $\mu|\Pi_s^\ast(x-\xi)|^2+\nu|\Pi_u^\ast(x-\xi)|^2$. Thus
(\ref{subsolution}) does hold if $0<\mu<1/C_1$ and $0<\nu<1/C_1$.
\end{proof}

Next we prove

\begin{lem} If $\Phi_\mu^\nu(x)$ and $\mu$, $\nu$ are as in
Lemma \ref{lemphimunu} then $W(x)>\Phi_\mu^\nu(x)$ in
$\overline{U^\prime(\xi)}\setminus\{\xi\}$, where
$U^\prime(\xi)\subset U(\xi)$ is a neighborhood of $\xi$.
\label{lemsubsol}
\end{lem}

\begin{proof} Since $\xi$ is a significant point,
we have $W(x)=d_{\overline H}(x,\xi)$ in $U(\xi)$. Due to the
representation formula (\ref{distancefunction1}) $d_{\overline
H}(x,\xi)=\lim_{k\to \infty}\int_0^{t^k}\overline L(\dot
\eta^k,\eta^k)\,{\rm d}\tau$ for a sequence $t^k>0$ and absolutely
continuous curves $\eta^k: [0,t^k]\to\overline\Omega$ satisfying
the initial and terminal conditions $\eta^k(0)=\xi$,
$\eta^k(t^k)=x$. We claim that there is a neighborhood
$U^\prime(\xi)\subset U(\xi)$ such that $\forall x\in
U^\prime(\xi)$ we have $\eta^k(\tau)\in U(\xi)$ $\forall \tau
\in[0,t^k]$ when $k$ is sufficiently large. Indeed, otherwise
there are sequences of points $x^k\to \xi$ and curves $\eta^k(t)$,
$\eta^k(0)=\xi$, $\eta^k(t^k)=x^k$ that exit $U(\xi)$ at a time
$t=\tau^k$ and $\lim_{k\to \infty}\int_0^{t^k}\overline L(\dot
\eta^k,\eta^k)\,{\rm d}\tau=0$. Let us set $y^k:=\eta^k(\tau^k)\in
\partial U^\prime(\xi)$, then $\lim_{k\to \infty}S_{\overline H}(y^k,\xi)=0$ and
after extracting a subsequence $y^k\to y\in\partial
U^{\prime}(\xi)$ we obtain $S_{\overline H}(y,\xi)=0$. Therefore
$y\in\mathcal{A}_{\overline H}$. We can repeat this reasoning to
find $y\in\mathcal{A}_{\overline H}\cap\partial
U^{\prime\prime}(\xi)$ for every open subset
$U^{\prime\prime}(\xi)$ of $U(\xi)$ containing point $\xi$. Thus
$\xi$ cannot be isolated point of $\mathcal{A}_{\overline H}$,
contradicting (\ref{fixedpoints}).

 Now using
(\ref{subsolution}) we get, for every $x\in U^\prime(\xi)$
\begin{multline*}
\Phi_\mu^\nu(x)=\int_0^{t^k}\nabla
\Phi_\mu^\nu(\eta^k)\cdot\dot\eta^k\,{\rm
d}\tau=\int_0^{t^k}\bigl(\nabla
\Phi_\mu^\nu(\eta^k)\cdot\dot\eta^k-\overline H(\nabla
\Phi_\mu^\nu(\eta^k),\eta^k)\bigr)\,{\rm d}\tau\\
+\int_0^{t^k} \overline H(\nabla \Phi_\mu^\nu(\eta^k),\eta^k) \,{\rm
d}\tau\leq \int_0^{t^k}\overline L(\dot \eta^k,\eta^k) \,{\rm d}\tau,
\end{multline*} when $k$ is sufficiently large. It follows that
$\Phi_\mu^\nu\leq W$ in $U^\prime(\xi)$. On the other hand if
$\Phi_\mu^\nu= W$ at a point $x_0\in U^\prime(\xi)$ then $x_0$ is
a local minimum of $W-\Phi_\mu^\nu$ and ${\overline
H}(\nabla\Phi_\mu^\nu(x_0),x_0)\geq 0$ since $W$ is a viscosity
solution of ${\overline H}(\nabla W(x),x)=0$ in $\Omega$.
Therefore $x_0=\xi$ by (\ref{subsolution}), i.e. $\Phi_\mu^\nu<W$
in $U^\prime(\xi)\setminus\{\xi\}$ and  by choosing, if necessary, a smaller
$U^\prime(\xi)$ we are done.
\end{proof}

The following is the crucial step in establishing
(\ref{specialclass}). We construct a test function $\Psi_\ve(x)$
of the form $\Psi_\ve(x)=\Phi_\mu^\nu(x)-\ve \tilde
\theta_\ve(x,x/\ve)$ which satisfies
\begin{equation}
 -\ve a^{ij}(x,x/\ve)\frac{\partial^2
\Psi_\ve}{\partial x_i\partial
x_j}+H\bigl(\nabla\Psi_\ve(x),x,x/\ve\bigr)\leq \overline H(\nabla
\Phi_\mu^\nu(x),x)+C\ve\qquad \text{in}\ U^\prime(\xi).
\label{subsolutunoform}
\end{equation}
Assuming first that the solution $\vartheta(p,x,y)$ of
(\ref{effproblresonans}), normalized by $\int_Y
\vartheta(p,x,y){\rm d} y=1$, is sufficiently smooth, we set
$\tilde \theta_\ve(x,y)=\theta\bigl(\nabla \Phi_\mu^\nu(x),x,y\bigr)$, where
$\theta(p,x,y) =\log \vartheta(p,x,y)$. Then, since
$$
-a^{ij}(x,y)\frac{\partial^2 \theta(p,x,y)}{\partial y_i\partial
y_j}+H\bigl(p+\nabla_y\theta(p,x,y),x,y\bigr)=\overline H(p,x),
$$
one easily shows (\ref{subsolutunoform}). Note that in this case
$\tilde \theta_\ve(x,y)$ is independent of $\ve$. In the general
case, thanks to $C^1$-regularity of the coefficients $a^{ij}(x,y)$
and $b^j(x,y)$, all the first and second order partial derivatives of
$\vartheta(p,x,y)$ exist and continuous on
$\mathbb{R}^N\times\overline{\Omega}\times \mathbb{R}^N $, except
(possibly) $\partial^2\vartheta(p,x,y)/\partial x_i\partial x_j$.
To obtain sufficient regularity of $\tilde \theta_\ve(x,y)$ we set
%
$$
\tilde\theta_\ve(x,y)= \int
\varphi_\ve(x-x^\prime)\theta\bigl(\nabla\Phi_\mu^\nu(x),x^\prime,y\bigr){\rm
d}x^\prime,
$$
where $\varphi_\ve(x)=\ve^{-N}\varphi(x/\ve)$, with $\varphi(x)$ being a smooth
compactly supported nonnegative function and $\int \varphi(x)\,{\rm
d}x=1$. Then we have
$$
a^{ij}(x,x/\ve)\Bigl(\frac{\partial^2 \theta}{\partial y_i\partial
y_j}\bigl(\nabla\Phi_\mu^\nu(x),x,x/\ve\bigr)- \frac{\partial^2}{\partial
x_i\partial x_j} \bigl(\theta_\ve(x,x/\ve)\bigr)\Bigr)\leq
\frac{C}{\ve}
$$
and $\bigl|\nabla_y\theta\bigl(\nabla\Phi_\mu^\nu(x),x,x/\ve\bigr)-\nabla
\bigl(\theta_\ve(x,x/\ve)\bigr)\bigr|\leq C$. This eventually leads to
(\ref{subsolutunoform}).

It follows from (\ref{subsolutunoform}) and (\ref{subsolution}) that
\begin{equation}
\label{finalsubsolution}
-\ve a^{ij}(x,x/\ve)\frac{\partial^2 \Psi_\ve}{\partial
x_i\partial x_j}+H\bigl(\nabla\Psi_\ve(x),x,x/\ve\bigr)\leq -\delta|x-\xi|^2+C\ve\qquad
\text{in}\ U^\prime(\xi).
\end{equation}
Consider now the function $W_\ve-\Psi_\ve$. By Lemma
\ref{lemsubsol} we have $W_\ve>\Psi_\ve$ on $\partial
U^\prime(\xi)$ for sufficiently small $\ve$, therefore either
$W_\ve\geq\Psi_\ve$ in $U^\prime (\xi)$ or $W_\ve-\Psi_\ve$
attains its negative minimum in $U^\prime(\xi)$ at a point
$x_\ve$. In the latter case we have $\nabla W_\ve(x_\ve)=\nabla
\Psi_\ve(x_\ve)$ and $a^{ij}(x_\ve,x_\ve/\ve)\frac{\partial^2
W_\ve}{\partial x_i\partial x_j}(x_\ve)\geq
a^{ij}(x_\ve,x_\ve/\ve)\frac{\partial^2 \Psi_\ve}{\partial
x_i\partial x_j} (x_\ve)$, this yields
\begin{multline*}
\lambda_\ve=-\ve a^{ij}(x_\ve,x_\ve/\ve)\frac{\partial^2
W_\ve}{\partial x_i\partial x_j}(x_\ve)+ H\bigl(\nabla
W_\ve(x_\ve),x_\ve,x_\ve/\ve\bigr)\\
\leq -\ve a^{ij}(x_\ve,x_\ve/\ve)\frac{\partial^2
\Psi_\ve}{\partial x_i\partial x_j}(x_\ve)+ H\bigl(\nabla
\Psi_\ve(x_\ve),x_\ve,x_\ve/\ve\bigr)\leq -\delta|x_\ve-\xi|^2+C\ve.
\end{multline*}
Thus either $W_\ve>\Psi_\ve$ in $U^\prime(\xi)$ or $W_\ve\geq
\Psi_\ve+W_\ve(x_\ve)-\Psi_\ve(x_\ve)$ in $U^\prime(\xi)$ and
$x_\ve$ satisfies $|x_\ve-\xi|\leq C\sqrt{\ve}$. Both cases lead
to the bound $W_\ve(x)\geq \Phi_\mu^\nu(x)+W_\ve(\tilde
x_\ve)-\beta\ve$, where $\tilde x_\ve$ is either $\xi$ or $x_\ve$
(recall that $u_\ve$ is normalized by $u_\ve(\xi)=1$, i.e.
$W_\ve(\xi)=0$). Then setting $z=(x-\xi)/\sqrt{\ve}$  we get
$$
w_\ve(z)\leq C w_\ve(z_\ve) e^{-\mu\phi_s(z)+\nu\phi_u(z)}\qquad\text{in}\ (U(\xi)-\xi)/\sqrt{\ve},
$$
where $z_\ve=(\tilde x_\ve-\xi)/\sqrt{\ve}$ and hence $|z_\ve|\leq
C$. Observe that since $z_\ve$ stay in a fixed compact as $\ve\to 0$, then $w_\ve(z_\ve)\leq C$ and in the
limit we therefore obtain
$$
w(z)\leq C e^{-\mu\phi_s(z)+\nu\phi_u(z)}\qquad \text{in}\
\mathbb{R}^N.
$$
It remains to note that $\phi_s(z)\geq \gamma_3 |\Pi_s^*z|^2$ and
$\phi_u(z)\leq \gamma_4 |\Pi_u^*z|^2$ for some $\gamma_3,\gamma_4
>0$. Hence $w(z)$ does satisfy (\ref{specialclass}), and by Lemma \ref{uniquness}
we have $\sigma=-2{\rm tr}(\Gamma Q)=\sigma(\xi)$. Thus
\begin{equation}
\label{optimalupperbound}
 \limsup_{\ve\to 0}\lambda_\ve/\ve=\sigma(\xi)\qquad \text{for every significant point $\xi$}.
\end{equation}

Inequalities (\ref{optimallowerbound}) and
(\ref{optimalupperbound}) prove formula (\ref{selection}).
Moreover they imply the uniqueness of the limiting additive
eigenfunction $W(x)$, provided that the maximum in
(\ref{selection}) is attained at exactly one $\xi=\overline\xi$.
Indeed, we know that, up to extracting a subsequence, functions
$W_\ve$ converge uniformly (on compacts in $\Omega$) to an
additive eigenfunction $W(x)$; here $W_\ve=-\ve\log u_\ve$ and
$u_\ve$ are referred to the eigenfunctions normalized by
(\ref{normalization}). By (\ref{optimalupperbound}) the unique
significant point (associated to the chosen subsequence) is
$\overline\xi$. Therefore $\overline \xi$ is the only minimal
element in $\mathcal{A}_{\overline H}$ with respect to the order
relation $\preceq$ defined in (\ref{orderrelation}); hence it is the least element
of $\mathcal{A}_{\overline H}$, i.e. $\overline \xi\preceq \xi$ for
every $\xi\in \mathcal{A}_{\overline H}$. This means that
$W(\xi)=W(\overline \xi)+d_{\overline H}(\xi,\overline\xi)$
$\forall \xi \in \mathcal{A}_{\overline H}$, and consequently
$W(x)=d_{\overline H}(x,\overline\xi)+W(\overline \xi)$. Finally,
by Corollary \ref{boundforminima} we have $W(\overline \xi)=0$, and
Theorem \ref{thwithoutdissipation} is now completely proved.
\hfill$\square$

\section{Other scalings}
\label{Sec8}

Theorem \ref{thwithoutdissipation} can be generalized to the case of $\ve^{\alpha}$-scaling, $\alpha>1$, of the
fast variable in (\ref{operator}) (with $c(x,y)=0$). In this case the effective drift
is still defined by formula (\ref{effectivedrift}) where $\theta^*$ is now the $Y$-periodic solution
of the equation $\frac{\partial^2}{\partial y_i\partial y_j} \bigl(a^{ij}(x,y)\theta^*(x,y)\bigr)=0$ normalized
by $\int_Y\theta^*\,{\rm d} y$. However, the discontinuous dependence of the effective drift on the parameter
$\alpha\geq 1$ (at $\alpha=1$) might lead to a significant shift of the concentration set
of the eigenfunction $u_\ve$ from points $\xi$ where $\overline{b}(\xi)=0$, if $\alpha>1$ sufficiently close to $1$.

We outline main changes to be made in order to adapt the  arguments
of Sections \ref{SecLowerbound} and \ref{SecUpperbound}
to the case $\alpha>1$.
First of all
let us introduce the approximate Hamiltonian $\overline{H}_\ve(p,x)$ as
the (additive) eigenvalue corresponding to a $Y$-periodic eigenfunction of
\begin{equation}
-a^{ij}(x,y)\frac{\partial^2 \theta_\ve(p,x,y)}{\partial y_i\partial
y_j}+H(p+\ve^{\alpha-1}\nabla_y\theta_\ve(p,x,y),x,y)=\overline H_\ve(p,x),
\label{epsilonHamiltonian}
\end{equation}
and the approximate drift $\overline{b}_\ve(x)$ by
\begin{equation*}
\overline{b}_\ve^{j}(x)=-\frac{\partial \overline{H}_\ve}{\partial p_j} (0,x).
\end{equation*}
The eigenvalue $\overline H_\ve$ is unique and $\theta_\ve$ is unique up to an additive constant,
moreover  $\theta_\ve$ can be found as the scaled logarithmic transformation
$\theta_\ve=-\frac{1}{\ve^{2(\alpha-1)}}\log\vartheta_\ve$ of a
positive $Y$-periodic eigenfunction of the problem
\begin{equation*}
\ve^{2(1-\alpha)}a^{ij}(x,y)\frac{\partial^2 \vartheta_\ve}{\partial y_i\partial
y_j}+\ve^{1-\alpha} (b^j(x,y)-2a^{ij}(x,y) p_i)\frac{\partial \vartheta_\ve}{\partial y_j}+
H(p,x,y)\vartheta_\ve=
\overline H_\ve(p,x)\vartheta_\ve.
\end{equation*}
Similarly to the case $\alpha=1$ the drift $\overline{b}_\ve(x)$ can be equivalently defined by
\begin{equation*}
\overline{b}_\ve(x)=\int_Y b(x,y)\theta^*_\ve(x,y)\, {\rm d}y,
\end{equation*}
via the $Y$-periodic solution $\theta^*_\ve$ of the equation
\begin{equation*}
\frac{\partial^2 }{\partial y_i\partial
y_j}\bigl(a^{ij}(x,y)\theta^\ast_\ve\bigr)- \ve^{\alpha-1}\frac{\partial }{\partial y_j}\bigl(b^j(x,y) \theta^\ast_\ve\bigr)=0.
\end{equation*}

It is clear that $\overline{b}_\ve\to\overline{b}$ in $C^1$ topology, therefore if $\overline{b}$ has,
say $n$, zeros in $\Omega$, where the hyperbolicity condition (for the ODE $\dot x=-\overline{b}(x)$) is satisfied, then $\overline{b}_\ve$ has exactly
$n$ zeros at the distance at most $O(\ve^{\alpha-1})$ from the corresponding zeros of $\overline{b}$. Then, to show the lower
bound for eigenvalues one follows
the lines of Section~\ref{SecLowerbound} with a zero $\xi_\ve$ of $\overline{b}_\ve$ in place of the corresponding zero $\xi$ of
$\overline b$, and $\theta^\ast_\ve$ in place of $\theta^\ast$. Note that although $\xi_\ve\to \xi$ as $\ve\to 0$, the
distance between this two points might be of order $\ve^{\alpha-1}$, so that in the local scale $\sqrt{\ve}$ this distance
tends to infinity. Nevertheless, up to the shift from $\xi$ to $\xi_\ve$ the local analysis is exactly the same as in
Section~\ref{SecLowerbound}.
Let us emphasize that for $\alpha\in (1,3/2)$ the statement of of Lemma~\ref{uniformpreciseboundeigenvalue}
remains valid only if at least one of zeros of $\overline b$ is an interior point of $\Omega$.

The argument of Section \ref{SecUpperbound}
can also be adapted to the case $\alpha>1$. As in the proof of the lower bound one finds
equation (\ref{homogequationrescaled1}) for the limit of rescaled functions $w_\ve(z)=u_\ve(\xi_\ve+\sqrt{\ve}z)$,
while the construction of the functions $\Phi^\nu_\mu$ and $\Psi_\ve$ is to be modified.
One can linearize the drift $\overline b_\ve$ at $\xi_\ve$ and construct the quadratic function  $\Phi^\nu_\mu$ (which now
depends on $\ve$) following Section~\ref{SecUpperbound} with $B^{ij}_\ve=\frac{\partial \overline{b}^j}{\partial x_i}(\xi_\ve)$
in place of $B^{ij}$; also, in the construction of the function $\Psi_\ve$ one makes use of  the eigenfunction $\theta_\ve$
(cf. (\ref{epsilonHamiltonian})) and sets
$\Psi_\ve(x) =\Phi^\nu_\mu(x) +\ve^{2\alpha -1}\theta_\ve(\nabla\Phi^\nu_\mu(x),x,x/\ve^\alpha)$. Details are left to the reader.

Finally note that the case $\alpha <1$ remains completely open.

\section{Example}

Here we consider an example of an operator of the form
(\ref{operator-redu}) for which conditions  (\ref{fixedpoints}) are fulfilled.  Let $x^y(t)$ be a solution of the ODE
$\dot x^y=-\overline b(x^y)$, $x^y(0)=y$.
We assume that
\begin{itemize}
\item The vector field $\overline b(x)$ has exactly three
zeros $\xi^1,\xi^2,\xi^3$ in
$\overline\Omega$. All of them are interior points of $\Omega$.
\item $\xi^1$ and $\xi^3$ are stable hyperbolic points,
that is eigenvalues of
$\bigl(-\frac{\partial\overline b^j}{\partial x_i}(\xi^1)\bigr)_{i,j=\overline{1,N}}$ and
$\bigl(-\frac{\partial\overline b^j}{\partial x_i}(\xi^3)\bigr)_{i,j=\overline{1,N}}$
have negative real parts; $\xi^2$ is a hyperbolic point and $\sigma(\xi^2)>\max\{\sigma(\xi^1),\sigma(\xi^3)\}$.
\item The ODE $\dot x=-\overline b(x)$ does not have
a solution with $\lim_{t\to+\infty}x(t)=\lim_{t\to-\infty}x(t)=\xi^2$.
\item For every
$y\in\overline\Omega\setminus\bigcup_{j=1}^3\{\xi^j\}$,
either $\lim\limits_{t\to-\infty}x^y(t)=\xi^2$, or $\inf\{t<0\,;\,x^y(t)\in\overline\Omega\}>-\infty$.
\end{itemize}
One can easily check that under these assumptions the Aubry set ${\mathcal A}_{\overline H}$ coincides with
$\bigcup_{j=1}^3\{\xi^j\}$.  Hence, by Theorem
\ref{thwithoutdissipation},  $W(x)=d_{\overline H}(x,\xi^2)$ and $\lambda_\ve=\ve\sigma(\xi^2)+\bar o(\ve)$.

It is interesting to trace in this example the possible structure of the set $\mathcal{Z}=\{x\in\Omega\,;\,W(x)=0\}$.
It depends on whether there are trajectories of the equation $\dot x=-\overline b(x)$ going from
$\xi^1$ or $\xi^3$ to $\xi^2$, or not.

Let $\mathcal{Z}^1$ be the set of all points $y\in \overline\Omega$ such that
$\lim\limits_{t\to+\infty}x^y(t)=\xi^1$ and $\lim\limits_{t\to-\infty}x^y(t)=\xi^2$, and
let
$\mathcal{Z}^3$ the set of all points $y\in \overline\Omega$ such that  $\lim\limits_{t\to+\infty}x^y(t)=\xi^3$ and $\lim\limits_{t\to-\infty}x^y(t)=\xi^2$. It follows from (\ref{distancefunction1}) that $\mathcal{Z}=\{\xi^2\}\cup\overline{\mathcal{Z}^1}\cup\overline{\mathcal{Z}^3}$.

\section{Appendices}
\appendix

\section{\large Uniqueness of additive eigenfunction}
\label{A}

The following simple result is a uniqueness criterion for
problem (\ref{homprob})-(\ref{hombcond}).

\begin{prop} \label{p_app1} Let $\lambda=\lambda_{\overline H}$
so that (\ref{homprob})-(\ref{hombcond}) has a solution $W$. Then
$W$ is unique (up to an additive constant) if and only if
$S_{\overline H-\lambda}(x,y)=0$ for all
$x,y\in\mathcal{A}_{\overline H-\lambda}$, where $S_{\overline
H-\lambda}(x,y)=d_{\overline H-\lambda}(x,y)+d_{\overline
H-\lambda}(y,x)$.
\end{prop}

\begin{proof} If $S_{\overline H-\lambda}(x,y)=0$
then $W(x)-W(y)=d_{\overline H-\lambda}(x,y)$, since $\forall x,y$
we have $W(x)-W(y)\leq d_{\overline H-\lambda}(x,y)$. In
particular, if $S_{\overline H-\lambda}(x,y)=0$ $\forall
x,y\in\mathcal{A}_{\overline H-\lambda}$ then pick $\xi\in
\mathcal{A}_{\overline H-\lambda}$ to get $W(x)=d_{\overline
H-\lambda}(x,\xi)+W(\xi)$ on $\mathcal{A}_{\overline H-\lambda}$.
Thus $W(x)=d_{\overline H-\lambda}(x,\xi)+W(\xi)$ in $\Omega$
(according to the representation formula (\ref{represnt_forH})),
i.e. $W$ is unique up to an additive constant.

If there are two points $\xi,\xi^\prime\in\mathcal{A}_{\overline
H-\lambda}$ such that $S_{\overline H-\lambda}(\xi,\xi^\prime)>0$,
then $W_0(x)=d_{\overline H-\lambda}(x,\xi)$ and
$W_1(x)=d_{\overline H-\lambda}(x,\xi^\prime)-d_{\overline
H-\lambda}(\xi,\xi^\prime)$ are two solutions of
(\ref{homprob})-(\ref{hombcond}) and $0=W_0(\xi)=W_1(\xi)$, while
$W_0(\xi^\prime)-W_1(\xi^\prime)=S_{\overline
H-\lambda}(\xi,\xi^\prime)>0$.
\end{proof}

\section{\large Aubry set for small perturbations of a gradient field}
\label{B}

We outline here the proof of the claim stated in Remark
\ref{perturbedAubry}. Consider a vector field $b(x,y)$ which is a
$C^1$-small perturbation of $\nabla P(x)$, i.e. $\|b(x,y)-\nabla
P(x)\|_{C^1}=\delta$
and $\delta$ is sufficiently small.  Let us show that the Aubry
set $\mathcal{A}_{\overline H}$ of the Hamiltonian $\overline
H(p,x)$ given by (\ref{effproblresonans}) (with $c(x,y)=0$) is
exactly the set of zeros of $\overline b(x)$ in $\Omega$,
provided that  $\delta$ is sufficiently small and $P\in C^2(\overline\Omega)$ is 
as in Remark \ref{perturbedAubry}.

Without loss of generality we can assume that $\overline
H(p,x)=\sum p_i^2- \overline{b}^i(x) p_i$, since the Aubry set of
this Hamiltonian coincides with that of the effective Hamiltonian
given by (\ref{effproblresonans}).  Let us first find the Aubry
set $\mathcal{A}_{H^0}$ of the Hamiltonian $H^0(p,x)=\sum p_i^2-
p_i\frac{\partial P(x)}{\partial x_i}$. We calculate the
corresponding Lagrangian $L^0(v,x)=\frac {1}{4}|v+\nabla P(x)|^2$
and use criterion (\ref{AubryVariational}). Let
$\xi\in\mathcal{A}_{H^0}$, then there exist a sequence of
absolutely continuous curves $\eta^n:[0,t^n]\to\overline\Omega$,
$\eta^n(0)=\eta^n(t^n)=\xi$, such that $t^n\to\infty$ and
$\lim_{n\to\infty} \int_0^{t^n}|\dot \eta^n+\nabla
P(\eta^n)|^2{\rm d}\tau=0$. This yields
$$
0=\lim_{n\to\infty} \int_0^{t^n}(|\dot \eta^n|^2+2\nabla
P_\delta(\eta^n)\cdot \dot \eta^n+|\nabla P(\eta^n)|^2){\rm
d}\tau=\lim_{n\to\infty}\int_0^{t^n}(|\dot \eta^n|^2+|\nabla
P(\eta^n)|^2){\rm d}\tau.
$$
Therefore, $\eta^n(t)\to \xi$ uniformly on every fixed interval
$[0,T]$. It follows that $\xi$ belongs to the set
$K=\{x\in\Omega;\ \nabla P(x)=0\}$.
Clearly, we also have $K\subset \mathcal{A}_{H^0}$. Now note that
the effective drift $\overline b(x)$ given by
(\ref{effectivedrift}), can be written as $\overline b(x)=\nabla
P(x)+\tilde b_\delta(x)$ with $C^1$-small $\tilde b_\delta(x)$,
$\|\tilde b_\delta \|_{C^1}=\overline O(\delta)$ as $\delta\to 0$.
Thanks to the assumption on critical points of $P(x)$, zeros of
$\overline b(x)$ are isolated and are close to $K$ when $\delta$
is sufficiently small. Moreover, if $\omega$ is a small
neighborhood of $\xi\in K$ then $\overline b(x)$ vanishes at
exactly one point $\xi_\delta \in \omega$ and
$|\xi-\xi_\delta|=O(\delta)$. Therefore, we can define a $C^2$
function $P_\delta$ such that $|\nabla P_\delta(x)|>0$ in
$\overline \Omega \setminus K_\delta$, where $K_\delta$ is the set
of zeros of $\overline b(x)$,  and $|\overline b(x)-\nabla
P_\delta (x)|=g_\delta(x) |\nabla P_\delta(x)|$ with
$\max_{x\in\overline\Omega} g_\delta(x)=\overline{O}(\delta)$ as
$\delta\to 0$. This yields the following  bound (for small
$\delta$)
$$
|v+\overline b(x)|^2\geq\frac{1}{2}|v|^2+2\nabla P_\delta(x)\cdot
v +V_\delta(x), \qquad \forall v\in \mathbb{R}^N, \ x\in \overline\Omega,
$$
where $V_\delta>0$ in $\overline\Omega\setminus K_\delta$. Then,
arguing as above 
we see that $\mathcal{A}_{\overline H}=K_\delta$. Moreover, every
 $\xi\in K_\delta$ is a hyperbolic fixed
point of the ODE $\dot x=-\overline b(x)$, as $\delta$ is sufficiently
small.

\section{\large Properties of solutions
of Bernoulli matrix equation}
\label{C}

We provide here some results on Bernoulli equation (\ref{Bernoulli}),
used in Sections \ref{SecLowerbound} and \ref{SecUpperbound}. Recall
that the matrix $Q$ in (\ref{Bernoulli}) is positive definite,
$\Pi_s$ and $\Pi_u$ denote spectral projectors on the invariant
subspaces of the matrix $B$ corresponding to eigenvalues with
positive and negative real parts.

\begin{prop}
\label{p_app2} The maximal positive semi-definite solution
$\Gamma$ of  (\ref{Bernoulli}) possesses the following properties:
(i) $\Gamma=\Pi_s \Gamma \Pi_s^\ast$, \ \ (ii)
$\Gamma\geq\gamma\Pi_s\Pi_s^\ast$ (in the sense of quadratic
forms) for some $\gamma>0$, \ \ (iii) $2{\rm tr} (Q\Gamma)={\rm
tr} (B\Pi_s)$, i.e. $2{\rm tr} (Q\Gamma)$ is the sum of positive
real parts of eigenvalues of $B$.
\end{prop}
\begin{proof}
It follows from  (\ref{Bernoulli}) that $X=\Pi_u\Gamma \Pi_u^\ast$
satisfies
\begin{equation}
4\Pi_u \Gamma Q \Gamma \Pi_u^\ast-(B\Pi_u) X - X(B\Pi_u)^\ast=0.
\label{BernForUnst}
\end{equation}
Consider the symmetric solution of (\ref{BernForUnst}) given by
\begin{equation}
\tilde X=\int_0^\infty Y(t)\,{\rm d} t, \label{ReprForUnst}
\end{equation}
where $Y(t)= -4 e^{B\Pi_u t}\Pi_u\Gamma Q \Gamma \Pi_u^\ast
e^{(B\Pi_u)^\ast t}$ (note that $\dot Y(t)=(B\Pi_u) Y(t) +
Y(t)(B\Pi_u)^\ast$ and $Y(t)\to 0$ as $t\to+\infty$, therefore
integrating we get $ (B\Pi_u) \tilde X + \tilde
X(B\Pi_u)^\ast=-Y(0)=4\Pi_u \Gamma Q \Gamma \Pi_u^\ast$, i.e.
$\tilde X$ does solve (\ref{BernForUnst})). We claim that
$X=\tilde X$. Otherwise $Z:=X-\tilde X$ is a nonzero solution of
equation $(B\Pi_u) Z + Z(B\Pi_u)^\ast=0$ and $Z=\Pi_u
Z\Pi_u^\ast$. Then $Z(t)=Z$ is a stationary solution of the
differential equation $\dot Z(t)=(B\Pi_u) Z(t) +
Z(t)(B\Pi_u)^\ast$. The latter equation has the solution $\tilde
Z(t)=e^{B\Pi_u t}\Pi_u Z \Pi_u^\ast e^{(B\Pi_u)^\ast t}$ which
vanishes as $t\to+\infty$ and satisfies the initial condition
$\tilde Z(0)=Z$. Thus $Z=0$, i.e. $X=\tilde X$. On the other hand
it follows from (\ref{ReprForUnst}) that $\tilde X\leq 0$ while
$X\geq 0$, this yields  $X=\tilde X =0$. Since $\Gamma$ is
positive semi-definite we also have $\Pi_u\Gamma= \Gamma
\Pi_u^\ast=0$ and the calculation
$\Gamma=(\Pi_u+\Pi_s)\Gamma(\Pi_u+\Pi_s)^\ast=\Pi_s \Gamma
\Pi_s^\ast$ shows (i). As a bi-product we also have established
that $\Gamma$ is the maximal positive semi-definite solution of
\begin{equation}
4\Gamma Q \Gamma -(B\Pi_s) \Gamma - \Gamma(B\Pi_s)^\ast=0
\label{BernoulliReduced}.
\end{equation}
Indeed, assuming that $\tilde \Gamma$ is another positive
semi-definite solution of (\ref{BernoulliReduced}) we get
$\Pi_u\tilde\Gamma Q \tilde\Gamma\Pi_u^{\ast}=0$. This yields
$\tilde\Gamma\Pi_u^{\ast}=0$ so that $\tilde\Gamma=\Pi_s
\tilde\Gamma \Pi_s^\ast$, therefore $B\tilde \Gamma=B(\Pi_s)^2
\tilde\Gamma \Pi_s^\ast=(B\Pi_s) \tilde\Gamma$ and $\tilde\Gamma$
thus solves (\ref{Bernoulli}).

To show (ii) and (iii) consider the maximal positive definite
solution $\tilde \Gamma_\delta$ of
\begin{equation}
4\tilde \Gamma_\delta Q \tilde \Gamma_\delta -(B\Pi_s+\delta I )
\tilde \Gamma_\delta - \tilde \Gamma_\delta(B\Pi_s+\delta
I)^\ast=0 \label{ApproxBernoul}
\end{equation}
for $\delta >0$. The existence of the unique positive definite
solution follows from the fact that $\tilde \Gamma_\delta^{-1}$ is
the unique solution of the Lyapunov matrix equation
\begin{equation}
4Q -\tilde \Gamma_\delta^{-1}(B\Pi_s+\delta I )  - (B\Pi_s+\delta
I)^\ast\tilde \Gamma_\delta^{-1}=0 \label{AuxLyapunov}
\end{equation}
given by
\begin{equation*}
\tilde \Gamma_\delta^{-1}=4\int_{-\infty}^0 e^{(B\Pi_s+\delta I
)^\ast t}Qe^{(B\Pi_s+\delta I)t}\,{\rm d} t.
\end{equation*}
It is known (see \cite{LR}) that $\tilde \Gamma_\delta$ converges
to the (maximal positive semi-definite) solution $\Gamma$ of
(\ref{BernoulliReduced}) as $\delta\to +0$. This allows to
establish (iii) easily,
$$
2{\rm tr} (Q\Gamma)=2 \lim_{\delta\to +0}2{\rm tr} (Q\tilde
\Gamma_\delta)=\frac{1}{2}\lim_{\delta\to +0}{\rm tr}\left(\tilde
\Gamma_\delta^{-1}(B\Pi_s+\delta I ) \tilde \Gamma_\delta +
(B\Pi_s+\delta I)^\ast\right) =  {\rm tr} (B\Pi_s).
$$
Finally, if we assume that (ii) is false, then there is $\eta\in \mathbb{R}^N$
such that $\Gamma \eta=0$ while $\Pi_s^\ast\eta\not =0$. Thanks to
(i) the equality $\Gamma \eta=0$ implies that $\Gamma \Pi^\ast_s\eta=0$.
On the other hand, $\Gamma
\left(\lim_{\delta\to +0}\tilde
\Gamma_\delta^{-1}\Pi_s\Pi_s^\ast\eta\right)=\Pi_s\Pi_s^\ast
\eta$, where the limit $\lim_{\delta\to +0}\tilde
\Gamma_\delta^{-1}\Pi_s\Pi_s^\ast\eta$ exists, for
$e^{(B\Pi_s+\delta I )^\ast t}Qe^{(B\Pi_s+\delta
I)t}\Pi_s\Pi_s^\ast \eta$ decays exponentially fast as
$t\to-\infty$, uniformly in $\delta\geq 0$. According the Fredholm
alternative $\Pi_s\Pi_s^\ast\eta$ and $\Pi_s^\ast\eta$ must be
orthogonal, yielding $|\Pi_s^\ast\eta|=0$. We obtained a
contradiction showing that (ii) does hold.
\end{proof}

\bigskip

\noindent
{\bf Acknowledgements.}
Part of this work was done when V. Rybalko was visiting the Narvik University College.
He is grateful for a warm hospitality and support of his visit.

\bigskip

\end{document}